\documentclass[reqno,11pt]{amsart}
\numberwithin{equation}{section}
\usepackage{amsmath}
\usepackage{esint} % permette di fare l'integrale tagliato
\usepackage{amsthm}
\usepackage{epsfig}
\usepackage{psfrag}
\usepackage{graphicx}
\usepackage{graphpap,latexsym,epsf}
\usepackage{color}
\usepackage{amssymb,mathrsfs,enumerate}
\definecolor{citegreen}{rgb}{0,0.6,0}
\definecolor{refred}{rgb}{0.8,0,0}
\usepackage[colorlinks, citecolor=citegreen, linkcolor=refred]{hyperref}
\newcommand{\R}{\mathbb{R}}

\newcommand{\Sph}{\mathbb{S}}
\def\DDD{{\rm D}}
\def\GGG{{\rm G}}
\def\HHH{{\rm H}}
\def\RRR{{\mathrm R}}

\def\a{\alpha}
\def\b{\beta}

\newcommand{\pa}{\partial}
\newcommand{\Om}{\Omega}
\newcommand{\ffi}{\varphi}

\newcommand{\ep}{\varepsilon}
\newcommand{\rmd}{{\rm d}}

\newcommand{\go}{g_0}
\newcommand{\cgo}{g^{(0)}} 
\newcommand{\ccgo}{g_{(0)}} 

\newcommand{\Ric}{{\rm Ric}}
\newcommand{\cRic}{{\rm R}^{(0)}}

\newcommand{\D}{{\rm D}}
\newcommand{\DD}{{\rm D}^2}
\newcommand{\De}{\Delta}

\newcommand{\cho}{{\rm h}^{(0)}}
\newcommand{\Ho}{{\rm H}}

\newcommand{\g}{g}

\newcommand{\Ricg}{{\rm Ric}_g}
\newcommand{\cRicg}{{\rm R}^{(g)}}
\newcommand{\Rg}{{\rm R}_g}
\newcommand{\na}{\nabla}
\newcommand{\nana}{\nabla^2}
\newcommand{\Deg}{\Delta_g}

\newcommand{\chg}{{\rm h}^{(g)}}
\newcommand{\Hg}{{\rm H}_g}

\mathchardef\emptyset="001F

\definecolor{vgreen}{rgb}{0.1,0.5,0.2}
\definecolor{viola}{RGB}{85,26,139}
\renewcommand{\theequation}{\thesection.\arabic{equation}}

\newtheorem{theorem}{Theorem}[section]
\newtheorem{remark}{Remark}
\newtheorem{corollary}[theorem]{Corollary}
\newtheorem{definition}{Definition}
\newtheorem{proposition}[theorem]{Proposition}

\newtheorem{lemma}[theorem]{Lemma}

\begin{document}

\hyphenation{ma-ni-fold}

\title[On the geometry of the level sets of bounded static potentials
]{On the geometry of the level sets of \\ bounded static potentials
}

\author[V.~Agostiniani]{Virginia Agostiniani}
\address{V.~Agostiniani, SISSA, via Bonomea 265, 34136 Trieste, Italy} 
\email{vagostin@sissa.it}

\author[L.~Mazzieri]{Lorenzo Mazzieri}
\address{L.~Mazzieri, Scuola Normale Superiore di Pisa,
Piazza Cavalieri 7, 56126 Pisa, Italy}
\email{l.mazzieri@sns.it}

% \thanks{}

\begin{abstract} 

In this paper we present a new approach to the study of asymptotically flat static metrics arising in general relativity. In the case where the static potential is bounded, we introduce new quantities which are proven to be monotone along the level set flow of the potential function. We then show how to use these properties to detect the rotational symmetry of the static solutions, deriving a number of sharp inequalities. As a consequence of our analysis, a simple proof of the classical $3$-dimensional Black Hole Uniqueness Theorem is recovered and some geometric conditions are discussed under which the same statement holds in higher dimensions.

\end{abstract}

\maketitle

\noindent\textsc{MSC (2010): 
35B06,
%PDE - symmetries, invariants of pdes
\!53C21,
%methods of Riem Geom (including pdes method)
\!83C57,
%black holes
\!35N25.
%overdetermined bvp
}

\smallskip
\noindent\keywords{\underline{Keywords}: static metrics, splitting theorem, Schwarzschild solution, overdetermined boundary value problems.} 

\date{\today}

\maketitle

%%%%%%%%%%%%%%%%%%%%%%%%%%%%%%%%%
%%%%%%%%%%%%%%%%%%%%%%%%%%%%%%%%%

\section{Introduction and statements of the results}
\label{sec:intro}

\bigskip

\subsection{Setting of the problem.}
Throughout this paper we let $(M, g_0)$ be an asymptotically flat $n$-dimensional Riemannian manifold, $n \geq 3$, with one end and nonempty smooth compact boundary $\pa M$, which is a priori allowed to have several connected components. We also assume that there exists a function $u \in {\mathscr C}^\infty  (M)$ such that the triple $(M, g_0, u)$ satisfies the system 
\medskip
\begin{equation}
\label{eq:pb_static_vacuum}
\left\{
\begin{array}{rcll}
\displaystyle
u\,\Ric\!\!\!\!&\, = \, &\!\!\!\!\DD u & 
{\rm in }\quad M,\\
\displaystyle
\De u\!\!\!\!& \, = \, &\!\!\!\!0  &
{\rm in }\quad M,\\
u \!\!\!\!& \, = \, &\!\!\!\! u_0   & {\rm on }\quad \pa M,\\
u(x) \!\!\!\!\!& \, \rightarrow&\!\!\!\! 1  & {\rm as }\quad |x| \to  +\infty,
% \displaystyle
% u\!\!\!\!&=&\!\!\!\!0  &{\rm on }\ \ \pa\Om,\\
%\displaystyle
%u\!\!\!\!&\to&\!\!\!\!1 & \mbox{at the end of }\ M.
\end{array}
\right.
\end{equation}
\smallskip

\noindent where $\Ric$, $\D$, and $\Delta$ represent the Ricci tensor, the Levi-Civita connection, and the Laplace-Beltrami operator of the metric $g_0$, respectively. Here, for simplicity, we let $u_0$ be a constant in $[0,1)$, but most part of the results can be easily adapted to the case where $u_0$ is a smooth function defined on the boundary of $M$ and tacking values in $[0,1)$. We notice that the first two equations in~\eqref{eq:pb_static_vacuum} are assumed to be satisfied in the whole $M$ in the sense that they hold in $M \setminus \pa M$ in the classical sense and if we take the limits of both the left hand side and the right hand side, they coincide at the boundary. In the rest of the paper the metric $g_0$ and the function $u$ will be referred to as {\em static metric} and {\em static potential}, respectively, whereas the triple $(M,g_0, u)$ will be called a {\em static solution}. A classical computation shows that if $(M,\go,u)$ satisfies~\eqref{eq:pb_static_vacuum}, then the Lorentzian metric $\gamma = -u^2 \, dt \otimes dt + \go$ satisfies the {\em vacuum Einstein equations}
\smallskip
\begin{equation*}
\Ric_\gamma \, = \, 0 \,  \quad \hbox{ in \,\, $\R \times (M \setminus \pa M)$} \, .
\end{equation*}

\smallskip

\noindent To complete the picture, we observe that, as a consequence of the system~\eqref{eq:pb_static_vacuum}, the scalar curvature $\RRR$ of $g_0$ is identically equal to zero. Moreover, in the special case where $u_0 = 0$, one has that the boundary $\pa M$ is a totally geodesics hypersurface embedded in $M$,
and the function $|\D u|$ is constant on each connected component of $\pa M$. It is also worth noticing that, since $u$ is a non constant harmonic function in $M$ and the boundary $\pa M$ is assumed to be regular, the Hopf Lemma implies that 
$|\D u|>0$ on $\pa M$. To further specify our assumptions, we recall the following definition from~\cite{Bun_Mas}.

\newpage

\begin{definition}[Asymptotically Flat Static Solutions]
\label{ass:AF} 
A solution $(M,g_0,u)$ to~\eqref{eq:pb_static_vacuum} is said to be {\em asymptotically flat
%{of order $\tau$, for some $\tau>(n-2)/2$}}, 
with one end} if there exists a compact set $K \subset M$
and a diffeomorphism $x = (x^1,...,x^n) :  M \setminus K \, \rightarrow \R^n \setminus B $ such that the metric $g_0$ and the static potential $u$ satisfy the following asymptotic expansions.
\begin{itemize}
\item[(i)] In the coordinates induced by the diffeomorphism $x$ the metric $\go$ can be expressed in $M \setminus K$ as 
$$
\go =  \cgo_{\alpha \beta} \, dx^\alpha \!\otimes dx^\beta ,
$$ 
and the components satisfy the decay conditions
\begin{equation}
\label{eq:AF}
\cgo_{\alpha \beta}=\delta_{\alpha \beta}+\eta_{\alpha \beta} \,,
\qquad\quad\mbox{with}\qquad\quad
\eta_{\alpha\beta} \, = \, o_2\big(|x|^{\frac{2-n}{2}}\big) \, ,\quad \hbox{as $|x| \to + \infty$} \, ,
\end{equation}
for every $\a, \b \in \{1,\!....,n \}$.

\smallskip

\item[(ii)] In the same coordinates, the static potential $u$ can be written as
\begin{equation}
\label{eq:uexp}
u \,\, = \,\, 1  -  m |x|^{2-n} + w \, ,
\qquad\mbox{with}\qquad \,\,
w \, = \,o_2(|x|^{2-n}) \, , \quad \hbox{as $|x| \to + \infty$} \, ,
\end{equation}
for some real number $m \in \R$. 
\end{itemize}
Here and thoughout the paper, we agree that for $f \in {\mathscr C}^{\infty}(M)$, $\tau \in \R$ and $k \in \mathbb{N}$ it holds
\begin{equation*}
f \, = \, o_k(|x|^{-\tau}) \quad \iff \quad \sum_{|J|\leq k} |x|^{\tau + |J|} \, \big|\pa^J \!f\big| \, = \, o(1) \, ,\quad\quad \hbox{as $|x| \to + \infty$} \,,
\end{equation*}
where the $J$'s are multi-indexes.
\end{definition}

To proceed, we consider an {\em asymptotically flat static solution} $(M,g_0,u)$ and we observe that since the function $u$ is harmonic and satisfies 
$$
u=u_0 \in [0, 1) \quad \hbox{on $\pa M$} \qquad \qquad \hbox{and} \qquad \qquad u(x) \to 1 \quad \hbox{as $x\to \infty$}\,, 
$$
it follows from the Strong Maximum Principle that $u_0<u<1$ in $M\setminus \pa M$. In particular, the coefficient $m$ that appears in the expansion~\eqref{eq:uexp} must be positive. It is a nontrivial consequence of~\eqref{eq:pb_static_vacuum} 
(see for instance~\cite{Bun_Mas} and \cite{Mia}) that such 
coefficient $m$ coincides with the ADM mass of the manifold $(M,g_0)$. 

By far, the most important solution to system~\eqref{eq:pb_static_vacuum} obeying the conditions of Definition~\ref{ass:AF} is the so called {\em Schwarzschild solution}. To describe it, we consider, for a fixed $m>0$, the manifold  with boundary $M$ given by the exterior domain $\R^n \setminus \{ |x| < (2m)^{1/(n-2)}\}$ in the flat Euclidean space, so that $\pa M  = \{ |x| = (2m)^{1/(n-2)}\}$. The {\em static metric} $\go$ and the {\em static potential} $u$ corresponding to the Schwarzschild solution are then given by
\begin{equation}
\label{eq:sol_schwarz}
g_0=\frac{d|x|\otimes d|x|}{\left(1- {2m}\, {|x|^{2-n}}\right)}
+|x|^2g_{\Sph^{n-1}}
\qquad \hbox{and}\qquad
u=\sqrt{1- {2m} \, {|x|^{2-n}}} \, ,
\end{equation}
respectively. The parameter $m>0$ is the ADM mass of the {\em Schwarzschild solution}. In dimension $n=3$, it is known by the work of Israel~\cite{Isr}, Robinson~\cite{Rob}, and
Bunting and Masood-Ul-Alam~\cite{Bun_Mas} 
that if one imposes the further condition $u\equiv0$ at $\pa M$, then~\eqref{eq:sol_schwarz} is the only {\em static solution} which is {\em asymptotically flat} with ADM mass equal to $m>0$. This is the content of the so called Black Hole Uniqueness Theorem (see~\cite{Chr_Cos_Heu,Hol_Ish,Rob_four_decades} for a comprehensive description of the subject).
% In particular the very elegant proof by Bunting and Masood-Ul-Alam does not assume the connectedness of the boundary of $M$ so that this property is deduced as a consequence of the rigidity statement. Unfortunately their argument is based on the Positive Mass Theorem, which so far it is known to hold up to dimension $n=7$. Nevertheless, a careful adaptation of their proof \cite{Gib_Ida_Shi} can be used to extend the result to the general $n$-dimensional case, at least for spin manifolds. 

In Subsection~\ref{sub:further}, we will discuss some consequence of our analysis in the special case where the {\em static potential} satisfies null Dirichlet boundary conditions at $\pa M$. In particular, when the boundary of $M$ is connected, or more in general when it is contained in a level set of $|\D u|$, we will recover the $3$-dimensional Black Hole Uniqueness Theorem (see Theorem~\ref{thm:bhu} below). Also, we will discuss some geometric conditions under which the same statement holds true in every dimension (see Theorem~\ref{thm:bhun} below).

\subsection{Statements of the main results.}
To introduce the main results, we start by the simple observation that, given a {static solution} $(M,g_0,u)$ to problem~\eqref{eq:pb_static_vacuum}, the function $U_1 \, : \, [u_0,1) \, \longrightarrow \, \R$ defined by
\begin{equation*} 
t \,\, \longmapsto \,\, U_1(t) \, = \!\!\! \int\limits_{ \{ u = t \}} \!\!\!\! |\D u| \, \rmd \sigma \phantom{\quad\,}
\end{equation*}
is constant, as it can be easily checked using the equation $\Delta u = 0$ and the Divergence Theorem. When $u_0 = 0$, such constant coincides with the capacity of the hypersurface $\pa M$ inside $(M, g_0)$, which, according to~\cite{Bra_Mia}, is defined as 
\begin{equation*}
{\rm Cap} (\pa M, g_0) = \inf  \bigg\{  \int_{\!M} \! |\D w|^2 \rmd \mu \,\, \Big| \,\,  \hbox{$w \in {\rm Lip}_{loc}(M)$, $w=0$ on $\pa M$,  $w \to 1$ as $|x| \to + \infty$}   \bigg\} \, ,
\end{equation*}
up to a multiplicative constant.
On the other hand, using the asymptotic expansions~\eqref{eq:AF} and~\eqref{eq:uexp} of $g_0$ and $u$, the constant value of $U_1$ can be computed in terms of the ADM mass $m>0$ of the {\em static solution} as
\begin{equation}
\label{eq:U1_mass}
 U_1(t) \, = \, || \D u ||_{L^1(\{u=t \})} \,\equiv \,  m\, (n-2)\, |\Sph^{n-1}|  \, , \quad\qquad t \in [u_0, 1) \, ,
\end{equation}
where $|\Sph^{n-1}|$ denotes the hypersurface area of the unit sphere sitting inside $\R^n$. 
Having this in mind, we introduce, for $p \geq 0$ and for a given constant Dirichlet boundary condition $u_0 \in [0, 1)$, the functions $U_p  :  [u_0,1) \, \longrightarrow \, \R $, defined as
\begin{equation}
\label{eq:Up}
t \,\, \longmapsto \,\, U_p(t) \, = \, \Big(\frac{2m}{1-t^2}\Big)^{\!\!\!\frac{(p-1)(n-1)}{(n-2)}}\!\!\!\!\!\! \int\limits_{ \{ u = t \}} \!\!\!\!  |\D u|^p \, \rmd \sigma .
\end{equation}
Formally, these functions can be thought of as renormalized $p$--capacities. In particular, we have that $t \mapsto U_0(t)$ is a renormalized hypersurface area functional for the level sets of $u$, whereas $t \mapsto U_1(t)$ is always constant, as already observed. In analogy with~\eqref{eq:U1_mass}, one can use the asymptotic expansions of $g_0$ and $u$ to deduce that
\begin{equation}
\label{eq:limup}
\lim_{t\to 1^{-}} \, U_p(t) \,\, = \,\, m^p (n-2)^p \, |\Sph^{n-1}| \, .
\end{equation}
Before proceeding, it is worth noticing that the functions $t \mapsto U_p(t)$ are well defined, since the integrands are globally bounded and the level sets of $u$ have finite hypersurface area. In fact, since $u$ is harmonic, the level sets of $u$ have locally finite $\mathscr{H}^{n-1}$-measure (see~\cite{Hardt_Simon} and~\cite{Lin}). Moreover, by the properness of $u$, they are compact and thus their hypersurface area is finite. Another important observation comes from the fact that, using the explicit formul\ae~\eqref{eq:sol_schwarz}, one easily realizes that the quantities
\begin{equation}
\label{eq:Pfunct_confvol}
M \, \ni \,x \, \longmapsto \,\,  \Big( \frac{2m}{1-u^2} \Big)^{\!\frac{n-1}{n-2}} \, |\DDD u| \,(x) \quad \hbox{and} \quad  [0, 1) \, \ni \, t \, \longmapsto \,\, U_0(t) \,= \!\!\!\!\int\limits_{\{u = t\}}\!\!
 \Big(\frac{1-u^2}{2m}\Big)^{\!\frac{n-1}{n-2}}
 \rmd\sigma 
\end{equation}
are constant on a Schwarzschild solution. In the following, via a conformal reformulation of problem~\eqref{eq:pb_static_vacuum}, we will be able to give a more geometric interpretation of this fact (see Subsections~\ref{sub:strategy} and~\ref{sub:conf}). On the other hand, we notice that the function $t \mapsto U_p(t)$ can be rewritten in terms of the above quantities as
\begin{equation}
\label{eq:Up2}
{ U}_p(t)\,\,=\!\!\!
\int\limits_{\{u = t\}}\!\!\!
\left[  \, \Big( \frac{2m}{1-u^2} \Big)^{\!\frac{n-1}{n-2}} \, |\DDD u|   \, \right]^{p}  \Big(\frac{1-u^2}{2m}\Big)^{\!\frac{n-1}{n-2}}
\,\, \rmd\sigma .
\end{equation}
Hence, thanks to~\eqref{eq:Pfunct_confvol}, we have that for every $p \geq 0$ the function $t \mapsto U_p(t)$ is constant on a Schwarzschild solution. Our main result illustrates how  the functions $t \mapsto U_p(t)$ can be used to detect the rotational symmetry of the {\em static solution} $(M, g_0, u)$. In fact, for $p\geq 3$, they are nonincreasing and the monotonicity is strict unless $(M, g_0, u)$ is isometric to a Schwarzschild solution. 

\newpage 

\begin{theorem}[Monotonicity-Rigidity Theorem]
\label{thm:main}
Let $(M,g_0,u)$ be an asymptotically flat solution to problem~\eqref{eq:pb_static_vacuum} in the sense of Definition~\ref{ass:AF}, with $0 \leq u_0<1$ and ADM mass equal to $m>0$. For every $p \geq 1$ we let $U_p : [u_0, 1) \rightarrow \R$ be the function defined  in~\eqref{eq:Up}. Then, the following properties hold true.
\begin{itemize}
\item[(i)] For every $p\geq 1$, the function $U_p$ is continuous.  

\smallskip

\item[(ii)] For every $p \geq 3$, the function $U_p$ is differentiable and the derivative satisfies, for every $t \in [u_0,1)$,
\begin{equation}
\label{eq:derup}
\phantom{\qquad}U_p'(t)
\,=\,
-\,(p-1)\, \Big(\frac{2m}{1-t^2}\Big)^{\!\!\!\frac{(p-1)(n-1)}{(n-2)}}\!\!\!\!\!
\int\limits_{\{u=t\}}
\!\!\!
|\DDD u|^{p-1}
\!\left[\, \HHH-
2\Big(\frac{n-1}{n-2}\Big)
\, \frac{u \,|\DDD u|}{1-u^2}
\, \right]  \rmd\sigma \,\, \leq \,\, 0 \, , 
\end{equation}
where $\HHH$ is the mean curvature of the level set $\{u=t\}$.
Moreover, if there exists $t \in [u_0, 1) \cap (0,1)$ such that $U_{p}'(t) = 0$ for some $p \geq 3$, then the static solution $(M,g_0,u)$ is isometric to a Schwarzschild solution with ADM mass equal to $m>0$. 

\smallskip

\item[(iii)] Suppose that $u=0$ at $\pa M$. Then $U_p'(0) = \lim_{t \to 0^+} U_p'(t)= 0$, for every $p\geq 3$. In particular, setting $U''_p(0) = \lim_{t \to 0^+}U_p'(t)/t$, we have that for every $p \geq 3$, it holds
\begin{equation}
\label{eq:der2up}
\phantom{\qquad} U_p''(0) \, = \,  - \Big(\frac{p-1}{2} \Big) \, (2m)^{\!\frac{(p-1)(n-1)}{(n-2)}} \!\!\!\int\limits_{\pa M}
\!
  |\DDD u|^{p-2}
\left[\, \RRR^{\pa M}\! - 
4 \, \Big(\frac{n-1}{n-2}\Big)
\, { \,|\DDD u|^2}
\, \right]  \rmd\sigma \, \leq \,0 \,  ,  
\end{equation}
where $\RRR^{\pa M}$ is the scalar curvature of the metric $g_{\pa M}$ induced by $g_0$ on $\pa M$.
Moreover, if $U_{p}''(0) =0$ for some $p \geq 3$, then the static solution $(M,g_0,u)$ is isometric to a Schwarzschild solution with ADM mass equal to $m>0$. 
%\begin{comm} 
%Probabilmente vale per $p \geq 2$ riflettere...
%\end{comm}
\end{itemize}
\end{theorem}
\begin{remark}
\label{rem:uno}
Notice that formula~\eqref{eq:derup} is well-posed also in the case where $\{u = t\}$ is not a regular level set of $u$. In fact, since $u$ is harmonic and proper, one has from~\cite{Hardt_Simon} and~\cite{Lin} that the $(n-1)$-dimensional Hausdorff measure of the level sets of $u$ is finite. Moreover, by the results in~\cite{Nadirashvili} and~\cite{Che_Nab_Val}, the Hausdorff dimension of its critical set is bounded above by $(n-2)$. In particular, the unit normal vector field to the level set is well defined $\mathscr{H}^{n-1}$-almost everywhere and so does the
mean curvature $\HHH$. In turn, the integrand in~\eqref{eq:derup} is well defined $\mathscr{H}^{n-1}$-almost everywhere. Finally, we observe that where $|\D u| \neq 0$ it holds
\begin{equation*}
|\D u|^{p-1} \HHH \,\, = \,\, - \, |\D u|^{p-4} \,\DD u (\D u , \D u) \,\, = \,\, - \, u \, |\D u|^{p-4} \,\Ric (\D u, \D u) \, .
\end{equation*}
Since $|\Ric|$ is uniformly bounded on $M$, this shows that the integrand in~\eqref{eq:derup} is essentially bounded and thus summable on every level set of $u$, provided $p \geq 2$. To conclude, we notice that also the hypersurface area element $\rmd\sigma$ is a priori well defined only on the regular portion of the level set. However, by the above arguments one can deduce that the density that relates $\rmd\sigma$ to the everywhere defined volume element $\rmd \mathscr{H}^{n-1}$ is well defined and bounded $\mathscr{H}^{n-1}$-almost everywhere on every level set of $u$. Hence, the integral in~\eqref{eq:derup} is well defined.
\end{remark}
\begin{remark}
Notice that under the hypothesis of the above theorem, formula~\eqref{eq:derup} implies that the only possible minimal level set is the one where $u$ vanishes, if present. 
\end{remark}
%\begin{remark}
%As it is clear from formula~\eqref{eq:derup}, in the case where $u=0$ at $\pa M$ we have that $U_p'(0)=0$, for every $p \geq 3$. Hence, the fact that $t=0$ is a critical point for $t \mapsto U_p(t)$ is not sufficient to deduce the rigidity statement. It turns that in this case the relevant condition is given by $U_p''(0) = \lim_{t \to 0^+} U_p'(t)/t =0$, as discussed in subsection~\ref{sub:further} and section~\ref{} below.
%\end{remark}
Before discussing the consequences of the above theorem (see Section~\ref{sec:conseq}) and giving some comments about the strategy of the proof, let us present a slight refinement of the main result, which holds on the end of $M$. To this aim, observe that the asymptotic behavior of the {\em static potential} $u$ forces the gradient $\D u$ to be nonzero outside of a fixed compact region. In particular, there exists a real number $\overline{u}_0 \in [u_0, 1)$ such that $|\D u|>0$ in the region $\{  \overline{u}_0 \leq u <1 \}$. As it will be clear from the proof of Theorem~\ref{thm:main}, this implies that the function $t \mapsto U_p(t)$ is continuous and differentiable in $(\overline{u}_0, 1)$ for every $p \geq 0$. Moreover, exploiting a refined version of the Kato inequality for harmonic functions, we deduce the same Monotonicity-Rigidity statement as in points $\rm (ii)$ and $\rm (iii)$ of  Theorem~\ref{thm:main} for a larger range of $p$'s. This is the content of the following theorem.

\begin{theorem}
\label{thm:refined}
Let $(M,g_0,u)$ be an asymptotically flat solution to problem~\eqref{eq:pb_static_vacuum} in the sense of Definition~\ref{ass:AF} with $0 \leq u_0<1$ and ADM mass equal to $m>0$. Let $\overline{u}_0 \in [u_0,1)$ be such that $|\D u|>0$ in the region $\{  \overline{u}_0 \leq u <1 \}$ and let $U_p : [u_0, 1) \rightarrow \R$ be the function defined  in~\eqref{eq:Up}. Then, the following properties hold true.
\begin{itemize}
\item[(i)] For $p\geq 0$, the function $U_p$ is continuous and differentiable in $(\overline{u}_0,1)$.  

\smallskip

\item[(ii)] For $p \geq 2-1/(n-1)$, we have that
$U_p'(t) \leq 0 $ for every $t \in (\overline{u}_0,1)$. 
Moreover, if there exists $t \in (\overline{u}_0, 1) $ such that $U_{p}'(t) = 0$ for some $p \geq 2-1/(n-1)$, then the static solution $(M,g_0,u)$ is isometric to a Schwarzschild solution with ADM mass equal to $m>0$.
\end{itemize}
\end{theorem}

\subsection{Strategy of the proof.} 
\label{sub:strategy}
To describe the strategy of the proof, we focus our attention on the rigidity statement (see Theorem~\ref{thm:main}-(iii)) and for simplicity, we let $p=3$. At the same time, we provide an heuristic for the the monotonicity statement.
The method employed 
%to prove Theorem~\ref{thm:main} and Theorem~\ref{thm:refined} 
is based on the conformal splitting technique introduced by the authors in~\cite{Ago_Maz_1}, which consists of two main steps. 
The first step is the construction of the so called {\em cylindrical ansatz} and amounts to find an appropriate conformal deformation $g$ of the {\em static metric} $g_0$ in terms of the {\em static potential} $u$. In the case under consideration, the natural deformation is given by
\begin{equation*}
%\label{eq:g_schw}
g \, = \, ({1-u^2})^{\frac2{n-2}} \, \go \,.
\end{equation*}
In fact, when $(M,g_0,u)$ is the {\em Schwarzschild solution}, the metric $g$ obtained through the above formula is immediately seen to be the cylindrical one. In general, the {\em cylindrical ansatz} leads to a conformal reformulation of  problem~\eqref{eq:pb_static_vacuum} in which the conformally related metric $g$ obeys the quasi-Einstein type equation
\begin{equation*}
\Ricg \, - \, \coth (\ffi) \nana\ffi \, + \, \frac{d\ffi\otimes d\ffi}{n-2}
\,\, = \,\,
\frac{|\na\ffi|^2_g}{n-2}\,\g \,,  \qquad \hbox{in $M$} ,
\end{equation*}
where $\na$ is the Levi-Civita connection of $g$ and the function $\ffi \, =\, \log\left[({1+u})/({1-u}) \right]$ is harmonic with respect to the Laplace-Beltrami operator of the metric $g$, namely
\begin{equation*}
\Delta_g \, \ffi \, = \, 0 \, , \quad \hbox{in $M$} .
\end{equation*}
Before proceeding, it is worth pointing out that taking the trace of the quasi-Einstein type equation gives
\begin{equation*}
\frac{\RRR_g}{n-1} \,= \, \frac{|\na \ffi|_g^2}{n-2} \, ,
\end{equation*}
where $\RRR_g$ is the scalar curvature of the conformal metric $g$. On the other hand, it is easy to see that $|\na \ffi|_g^2$ is proportional to the first term in~\eqref{eq:Pfunct_confvol}.
%, namely
%\begin{equation*}
% \Big( \frac{2m}{1-u^2} \Big)^{\!\frac{n-1}{n-2}} \, |\DDD u| \,(x) \,.
%\end{equation*}
In fact, if $(M,g_0)$ is a {\em Schwarzschild solution}, then $(M,g)$ is a round cylinder with constant scalar curvature. Furthermore, the second term appearing in~\eqref{eq:Pfunct_confvol} is (proportional to) the hypersurface area of the level sets of $\ffi$ computed with respect to the metric induced on them by $g$. Again, in the cylindrical situation such a function is expected to be constant.

The second step of our strategy consists in proving via a splitting principle that the metric $g$ has indeed a product structure, provided the hypotheses of the Rigidity statement are satisfied. More precisely, we use the above conformal reformulation of the original system combined with the Bochner identity to deduce the equation
\begin{equation*}
\Delta_\g|\na \ffi|_\g^{2}\, - \, \big\langle\na|\na \ffi|^{2}_\g \, \big| \,\na \log \big( \sinh( \ffi ) \big) 
\big\rangle_{\!\!\g} \, = \,  2  \,  
  \big|\nana \ffi\big|_\g^2
  \,.
\end{equation*}
Observing that the drifted Laplacian appearing on the left hand side is formally self-adjoint with respect to the weighted measure $(1/\sinh(\ffi)) \rmd \mu_g$, we integrate by parts and we obtain, for every $s\geq \ffi_0 = \log\left[({1+u_0})/({1-u_0}) \right]$, the integral identity
\begin{equation*}
\int\limits_{\{\ffi=s\}} \!\!
\frac{   |\na \ffi|_\g^{2}\,\HHH_g }{ \sinh(s)  }
\,\,\rmd\sigma_{\!g}  
\,\, =  \!\!
\int\limits_{\{\ffi> s\}} \!\!
\frac{   
\big|\nana \ffi\big|_\g^2
   }{\sinh(\ffi)}
\,\,\rmd\mu_g \, ,
\end{equation*}
where $\Hg$ is the mean curvature of the level set $\{ \ffi = s \}$ inside the ambient $(M,g)$ (notice that the same considerations as in Remark~\ref{rem:uno} apply here). We then observe that, up to a negative function of $s$, the left hand side coincides with $U_3'$ (see formul\ae\ \eqref{eq:upfip2} and~\eqref{eq:der_fip}), whereas the right hand side is always nonnegative. This implies the Monotonicity statement.
Also, under the hypotheses of the Rigidity statement, the left hand side of the  above identity vanishes and thus the Hessian of $\ffi$ must be zero in an open region of $M$. In turn, by analyticity, it vanishes everywhere. On the other hand, the asymptotic behavior of $u$ implies that $\ffi \to + \infty$ along the end of $M$. In particular, $\na \ffi$ is a nontrivial parallel vector field. Hence, it provides a natural splitting direction for the metric $g$. Finally, using the fact that $g_0$ is {\em asymptotically flat}, it is easy to realize that the asymptotics of $g$ are the ones of a round cylinder, so that the product structure forces the Riemannian manifold $(M,g)$ to be isometric to a round cylinder.

\subsection{Summary.}
The paper is organized as follows. In Section~\ref{sec:conseq} we describe the geometric consequences of Theorem~\ref{thm:main}, obtaining several sharp inequalities for which the equality is satisfied if and only if the solution to system~\eqref{eq:pb_static_vacuum} is rotationally symmetric. We distinguish the consequences of Theorem~\ref{thm:main}-(ii) on the geometry of a generic level set of $u$ (see Subsection~\ref{sub:rot_sym}), from the consequences of Theorem~\ref{thm:main}-(iii) on the geometry of the boundary of $M$ under null Dirichlet boundary conditions for $u$ (see Subsection~\ref{sub:further}). The results in these two subsections have a precise correspondence to each other so that, for example, Theorem~\ref{thm:main2_geom} corresponds to Theorem~\ref{thm:rig_gen}, Corollary~\ref{cor:Lp} corresponds to Corollary~\ref{cor:R_Lp}, Theorem~\ref{thm:ul_bounds} corresponds to Theorem~\ref{thm:ul_bounds_bound} and so on.
As it is evident from the statements of these theorems and corollaries, the role of the ratio $\HHH/(n-1)$ in Subsection~\ref{sub:rot_sym} is played in Subsection~\ref{sub:further} by the square route of the ratio $\RRR^{\pa M}/(n-1)(n-2)$, where $\RRR^{\pa M}$ is the scalar curvature of the metric $g_{|\pa M}$ induced by $g_0$ on the boundary. To illustrate this phenomenon, we observe that, in the framework of overdetermined elliptic boundary value problems, Corollary~\ref{cor:over} states that if the condition 
\begin{equation*}
\left(\frac{u}{1-u^2}\right) \,  \frac{\,\, 2 \, |\D u|\,\,}{n-2} \,\, = \,\, \frac{\HHH}{n-1}	
\end{equation*}
is satisfied on some level set of $u$, then the solution $(M,g_0,u)$ to system~\eqref{eq:pb_static_vacuum} must be rotationally symmetric. In the case where $u_0=0$, since the boundary of $M$ is totally geodesic, the above condition is always satisfied at $\pa M$ and thus does not imply in general any rigidity of the solution. The relevant overdetermining condition in this case is the one given in Corollary~\ref{cor:over2}, namely
\begin{equation*}
\left(\frac{\,2\, |\D u|\, }{n-2} \right)^{\!\!2} \,= \,\frac{\quad\RRR^{\pa M}}{(n-1)(n-2)}	\,.
\end{equation*}
In Subsection~\ref{sub:further}, assuming the connectedness of $\pa M$ we deduce a $n$-dimensional version of the Riemannian Penrose Inequality for static solutions (see Theorem~\ref{thm:ul_bounds_bound}-(iii)) as well as the classical $3$-dimensional Black Hole Uniqueness Theorem (see Theorem~\ref{thm:bhu}). We finally discuss in Theorem~\ref{thm:bhun} a geometric condition under which the uniqueness statement holds in every dimension $n \geq 4$.
For $n \geq 4$, we also derive Willmore-type inequalities for the level sets of $u$ in Subsection~\ref{sub:will}.

In Section~\ref{sec:conf_reform}, we reformulate problem~\eqref{eq:pb_static_vacuum} in terms of an asymptotically cylindrical quasi-Einstein type metric $g$ and a $g$-harmonic function $\ffi$ satisfying system~\eqref{eq:pb_schw_reform} ({\em cylindrical ansatz}), according to the strategy described in Subsection~\ref{sub:strategy}. In this context, Theorem~\ref{thm:main} and Theorem~\ref{thm:refined} are respectively equivalent to Theorem~\ref{thm:main_conf} and Theorem~\ref{thm:refined_conf} in Subsection~\ref{sub:reform} below. These latter statements will be proven in Section~\ref{sec:proofs} with the help of the integral identities obtained in Section~\ref{sec:integral}.

\subsection{Further directions.} Extending the ideas presented in~\cite{Ago_Maz_1},
one can develop a theory analogous to the one described in this paper in the case of classical potential theory, where 
%$(M, g_0)$ is isometric to an exterior domain $(\R^n \setminus \Omega, g_{\R^n})$ in the Euclidean space and  
problem~\eqref{eq:pb_static_vacuum} is replaced by 
\begin{equation*}
%\label{eq:pb}
\left\{
\begin{array}{rcll}
\displaystyle
\De u\!\!\!\!&=&\!\!\!\!0 & {\rm in }\quad\R^n\setminus\overline\Om,\\
\displaystyle
  u\!\!\!\!&=&\!\!\!\!1  &{\rm on }\ \ \pa\Om,\\
\displaystyle
u(x)\!\!\!\!&\to&\!\!\!\!0 & \mbox{as }\ |x|\to\infty,
\end{array}
\right.
\end{equation*}
and $\Omega$ is a bounded domain with regular boundary. For the sake of clearness, we decided to present these results in a separated paper. Another direction of research is to investigate possible applications of these ideas to the study of photon spheres in asymptotically flat static vacuum spacetimes, in the spirit of~\cite{Ced}.

%%%%%%%%%%%%%%%%%%%%%%%%%%%%%%%%%%%%%%%%%%%%%%%
%%%%%%%%%%%%%%%%%%%%%%%%%%%%%%%%%%%%%%%%%%%%%%%

\section{Consequences of the Monotonicity-Rigidity Theorem.}
\label{sec:conseq}

%%%%%%%%%%%%%%%%%%%%%%%%%%%%%%%%%%%%%%%%%%%%%%%
%%%%%%%%%%%%%%%%%%%%%%%%%%%%%%%%%%%%%%%%%%%%%%%

In this section we discuss some consequences of Theorem~\ref{thm:main}, distinguishing the general case (see Subsection~\ref{sub:rot_sym}) from the case where $u=0$ at $\pa M$ (see Subsection~\ref{sub:further}). It is worth noticing that, thanks to Theorem~\ref{thm:refined}, analogous corollaries hold on the end of $M$ for a larger range of $p$'s. However, for the sake of simplicity, we only describe the consequences of the Monotonicity-Rigidity Theorem~\ref{thm:main}. 
 
\subsection{Characterizations of the rotationally symmetric solutions.}
\label{sub:rot_sym}
Since, as already observed, the functions $t \mapsto U_p(t)$ defined in~\eqref{eq:Up} are constant on a Schwarzschild solution, we obtain, as an immediate consequence of Theorem~\ref{thm:main} and formula~\eqref{eq:derup}, the following characterizations of the rotationally symmetric solutions to  system~\eqref{eq:pb_static_vacuum}.
\begin{theorem}
\label{thm:main2_geom}
Let $(M,g_0,u)$ be an asymptotically flat solution to problem~\eqref{eq:pb_static_vacuum} in the sense of Definition~\ref{ass:AF} with $0 \leq u_0<1$ and ADM mass equal to $m>0$. Then, for every $p \geq 3$ and every $t \in [u_0, 1)$, the inequality 
\begin{equation}
\label{eq:int_ineq}
\frac{t}{1-t^2} \!\!\! \int\limits_{\{u=t\}} \!\!\!\frac{ \,\,2 \, |\D u|^p}{n-2} \, \rmd \sigma \,\,\, \leq  \!\int\limits_{\{u=t\}} \!\!\! |\D u|^{p-1} \frac{\HHH}{n-1} \, \rmd \sigma 
\end{equation}
holds true, where $\HHH$ is the mean curvature of the level set $\{u = t \}$. Moreover, the equality is fulfilled for some $p\geq 3$ and some $t \in [ u_0, 1) \cap (0,1)$ if and only if the static solution $(M,g_0,u)$ is isometric to a Schwarzschild solution with ADM mass equal to $m>0$.
\end{theorem}

To give an interpretation of Theorem~\ref{thm:main} in the framework of overdetermined boundary value problems, we observe that the equality is achieved in~\eqref{eq:derup} as soon as the term in square brackets vanishes $\mathscr{H}^{n-1}$-almost everywhere on some level set of $u$. This easily implies the following corollary.
\begin{corollary}
\label{cor:over}
Let $(M,g_0,u)$ be an asymptotically flat solution to problem~\eqref{eq:pb_static_vacuum} in the sense of Definition~\ref{ass:AF} with $0 \leq u_0<1$ and ADM mass equal to $m>0$. Assume in addition that the identity
\begin{equation}
\label{eq:overd}
\Big(\frac{u}{1-u^2}\Big) \,  \frac{\,\, 2 \, |\D u|\,\,}{n-2} \,\, = \,\, \frac{\HHH}{n-1}	
\end{equation}
holds $\mathscr{H}^{n-1}$-almost everywhere on some level set $\{ u = t\}$, with $t \in [ u_0, 1) \cap (0,1)$. Then, the static solution $(M,g_0,u)$ is isometric to a Schwarzschild solution with ADM mass equal to $m>0$.
\end{corollary}
In other words, assumption~\eqref{eq:overd} in the previous corollary can be seen as a condition that makes system~\eqref{eq:pb_static_vacuum} overdetermined and forces the solution to be rotationally symmetric. Observe that~\eqref{eq:overd} is always satisfied on a Schwarzschild solution and thus it is also a necessary condition for $(M,g_0, u)$ being rotationally symmetric.

To illustrate other implications of Theorem~\ref{thm:main2_geom}, let us observe 
that, applying H\"older inequality to the right hand side of~\eqref{eq:int_ineq} with conjugate exponents $p/(p-1)$ and $p$, one gets
\begin{equation*}
\int\limits_{\{u=t\}} \!\!\! |\D u|^{p-1} {\HHH} \, \rmd \sigma  \,\, \leq \,\, \Bigg(  \int\limits_{\{u=t\}} \!\!\! |\D u|^{p}  \, \rmd \sigma \Bigg)^{\!\!{(p-1)}/{p}} 
\Bigg(  \int\limits_{\{u=t\}} \!\!\! |\HHH|^{p}  \, \rmd \sigma \Bigg)^{\!\!{1}/{p}} \,.
\end{equation*}                             
This implies on every level set of $u$ the following sharp $L^p$-bound for the gradient of the {\em static potential} in terms of the $L^p$-norm of the mean curvature of the level set.
\begin{corollary}
\label{cor:Lp}
Let $(M,g_0,u)$ be an asymptotically flat solution to problem~\eqref{eq:pb_static_vacuum} in the sense of Definition~\ref{ass:AF} with $0 \leq u_0<1$ and ADM mass equal to $m>0$. Then, for every $p \geq 3$ and every $t \in [u_0, 1)$ the inequality
\begin{equation}
\label{eq:Lp_ineq}
\frac{t}{1-t^2} \,\,\,  \left|\left|  \frac{\, 2 \, \D u \, }{n-2} \right|\right|_{L^p(\{ u=t \})}  \!\!\!\!\! \leq \,\,\,\, \left|\left|  \frac{\HHH}{n-1} \right|\right|_{L^p(\{ u=t \})} , 
\end{equation}
holds true, where $\HHH$ is the mean curvature of the level set $\{u = t \}$. Moreover, the equality is fulfilled for some $p \geq 3$ and some $t \in [ u_0, 1) \cap (0,1)$ if and only if the static solution $(M,g_0,u)$ is isometric to a Schwarzschild solution with ADM mass equal to $m>0$.
\end{corollary}
It is worth pointing out that the right hand side in~\eqref{eq:Lp_ineq} may possibly be unbounded. However, for regular level sets of the {\em static potential} the $L^p$-norm of the mean curvature is well defined and finite (see Remark~\ref{rem:uno}). We also observe that letting $p \to + \infty$, we deduce, under the same hypothesis of Corollary~\ref{cor:Lp}, the following $L^\infty$-bound
\begin{equation}
\label{eq:Linf_ineq}
\frac{t}{1-t^2} \,\,\,  \left|\left|  \frac{\, 2 \, \D u\, }{n-2} \right|\right|_{L^\infty(\{ u=t \})}  \!\!\!\!\! \leq \,\,\,\, \left|\left|  \frac{\HHH}{n-1} \right|\right|_{L^\infty(\{ u=t \})}  ,
\end{equation}
for every $t \in [u_0, 1)$. Unfortunately, in this case we do not know whether the rigidity statement holds true or not. However, the equality is satisfied on a Schwarzschild solution with ADM mass equal to $m>0$ and this makes the inequality sharp.

We are now in the position to deduce sharp upper and lower bounds for the ADM mass $m>0$ of a {\em static solution} $(M,g_0,u)$. This will be done combining inequality~\eqref{eq:Lp_ineq} in Corollary~\ref{cor:Lp} with the simple observation that for every $t \in [u_0,1)$ and every $p \geq 3$ it holds
\begin{equation}
\label{eq:mon_limup}
U_p(t) \, \geq \, m^p (n-2)^p |\Sph^{n-1}| \,,  
\end{equation}
where the latter estimate follows immediately from~\eqref{eq:limup} and Theorem~\ref{thm:main}-(ii). To state the result, it is convenient to set, for every $f\in L^p(\{u=t \})$, 
\begin{equation}
\label{def:av_norm}
|| f ||_{L^p_0(\{ u= t\})} \, = \, \Bigg[ \frac{1}{|\{ u=t \}|} \int\limits_{\{ u=t\}} \!\!\! |f|^p \, \rmd \sigma  \Bigg]^{1/p} \, ,
\end{equation}
where $|\{ u=t \}|$ denotes the $\mathscr{H}^{n-1}$-measure of the level set $\{u=t \}$.
\begin{theorem}
\label{thm:ul_bounds}
Let $(M,g_0,u)$ be an asymptotically flat solution to problem~\eqref{eq:pb_static_vacuum} in the sense of Definition~\ref{ass:AF} with $0 \leq u_0<1$ and ADM mass equal to $m>0$. Then the following statements hold true.
\begin{itemize}
\item[(i)] For every $p \geq 3$ and every $t \in [u_0, 1)$ the inequalities 
\begin{equation}
\label{eq:pi_gen}
\phantom{\quad\quad} \frac{1-t^2}{2} \,\,  K(n,p,t) \,\,\left({   \frac{|\{u=t\}| }{|\Sph^{n-1}|} } \right)^{\!\!\frac{n-2}{n-1}} \!\! \leq  \,\,m \,\, \leq \,\,  \frac{1-t^2}{2t}  \,  \left|\left|   \frac{\HHH}{n-1} \right|\right|_{L^p_0(\{ u=t \})}  \left({   \frac{|\{u=t\}| }{|\Sph^{n-1}|} } \right) \, 
\end{equation}
hold true, where $|\{u=t\}|$ is the $\mathscr{H}^{n-1}$-measure of the level set of $u$ and 
$$
K (n,p,t) \, = \, \left[ \frac{ || \D u||_{L_0^1( \{u=t\})}   }{  || \D u||_{L_0^p(\{u=t\})}  } \right]^{\!\frac{p(n-2)}{(p-1)(n-1)}} \!\!\!.
$$
Moreover, the equality holds in either the first or in the second inequality, for some $t \in[u_0,1) \cap (0,1)$ and some $p \geq 3$, if and only if the static solution $(M,g_0,u)$ is isometric to a Schwarzschild solution with ADM mass equal to $m>0$. 

\smallskip

\item[(ii)] Assume that for some $t \in[u_0,1) \cap (0,1)$ and some $p \geq 3$ the mean curvature $\HHH$ of the level set $\{u=t \}$ satisfies the bound
\begin{equation}
\label{eq:H_bound}
 \left|\left|   \frac{\HHH}{n-1} \right|\right|_{L^p_0(\{ u=t \})} \leq \,\, t \,\, K(n,p,t) \, \left({   \frac{|\Sph^{n-1}|}{|\{u=t\}| } } \right)^{\!\!\frac{1}{n-1}}\, .
\end{equation}
Then the static solution $(M,g_0,u)$ is isometric to a Schwarzschild solution with ADM mass equal to $m>0$. 
\end{itemize}
\end{theorem}
\begin{remark}
The first inequality in~\eqref{eq:pi_gen} should be compared with the classical Riemannian Penrose Inequality, as described in Theorem~\ref{thm:ul_bounds_bound} below. In particular, the rigidity statement continues to hold 
for this inequality also when the equality is achieved on the level set $\{u=0\}$, if present.
\end{remark}
\begin{proof}
To obtain the first inequality in~\eqref{eq:pi_gen} it is sufficient to use~\eqref{eq:U1_mass} to rewrite the right hand side of~\eqref{eq:mon_limup} as $|| \D u||^p_{L^1(\{ u= t \})} |\Sph^{n-1}|^{1-p}$. This leads to 
\begin{equation}
\label{eq:L1_Lp}
{ || \D u||_{L^1_0( \{u=t\})}   }\,\, \leq \,\, 
\left[ \Big( \frac{2m}{1-t^2}  \Big)^{\!\frac{n-1}{n-2} }   \bigg(\frac{|\Sph^{n-1}|}{|\{ u= t\}|}\bigg) \, \right]^{\!\frac{p-1}{p}}  \!\! {  || \D u||_{L^p_0(\{u=t\})}   }   \, ,
\end{equation}
and, after some algebra, to the desired inequality. In this case, the rigidity statement follows from the fact that as soon as the equality is achieved in~\eqref{eq:mon_limup}, Theorem~\ref{thm:main} implies the rotational symmetry of the solution. To obtain the second inequality in~\eqref{eq:pi_gen}, we observe that, by Jensen Inequality one has that
\begin{equation*}
\int\limits_{\{u=t\}} \!\!\! |\D u|  \, \rmd \sigma \,\, \leq \,\, \big|\{ u=t \}\big|^{\frac{p-1}{p}} \,  || \D u||_{L^p(\{ u=t\})} \, . 
\end{equation*}
Using~\eqref{eq:U1_mass}, this can be rewritten as 
\begin{equation*}
m \,\, \leq \,\, \frac12 \,  \left|\left|  \frac{2 \, \D u}{n-2} \right|\right|_{L^p_0(\{ u=t \})} \!\! \left({   \frac{|\{u=t\}| }{|\Sph^{n-1}|} } \right) \, .
\end{equation*}
The desired inequality with the corresponding rigidity statement is now a straightforward consequence of Corollary~\ref{cor:Lp}. Finally, we observe that statement (ii) follows immediately from statement (i), since  inequality~\eqref{eq:H_bound} implies that the equality is fulfilled in~\eqref{eq:pi_gen}.
\end{proof}

\subsection{Further consequences under null Dirichlet boundary conditions.}
\label{sub:further}
We pass now to describe some consequences of Theorem~\ref{thm:main} in the case where the {\em static potential} satisfies null Dirichlet boundary condition at $\pa M$. In this case, one has that the boundary of $M$ is a totally geodesic hypersurface inside $(M,g_0)$ and $|\D u|$ is constant on every connected component of $\pa M$. In particular, also the mean curvature $\HHH$ vanishes at $\pa M$. Hence, formula~\eqref{eq:derup} implies that
%\begin{comm}
%for every $p \geq 3$,
%\end{comm}
$U_p'(0) = 0$. To illustrate the relevant condition for the rigidity statement in this case (see Theorem~\ref{thm:main}-(iii)), it is convenient to set $\nu = \D u / |\D u|$ and use the identity
%\begin{equation*}
$ |\D u| \, \HHH   =  - \, u  \, {\Ric (\nu, \nu)}$
%\end{equation*}
to rewrite~\eqref{eq:derup} as 
\begin{equation*}
%\label{eq:derup2}
U_p'(t)
\,=\,
\,(p-1) \,\, t \,\,  \Big(\frac{2m}{1-t^2}\Big)^{\!\!\!\frac{(p-1)(n-1)}{(n-2)}}\!\!\!\!\!
\int\limits_{\{u=t\}}
\!\!\!
|\DDD u|  ^{p-2}
\!\left[\, \Ric(\nu,\nu) + 
2 \, \Big(\frac{n-1}{n-2}\Big)
\, \frac{ \,|\DDD u|^2}{1-u^2}
\, \right]  \rmd\sigma \,\, \leq \,\, 0 \, . 
\end{equation*}
In particular, for every $p \geq 3$, one has that  
\begin{equation*}
U_p''(0) \, = \,  \lim_{t \to 0^+} \frac{U_p'(t)}{t} \, = \,  (p-1) \, (2m)^{\!\frac{(p-1)(n-1)}{(n-2)}} \!\!\!\int\limits_{\pa M}
\!
  |\DDD u|^{p-2}
\left[\, \Ric(\nu,\nu) + 
2 \, \Big(\frac{n-1}{n-2}\Big)
\, { \,|\DDD u|^2}
\, \right]  \rmd\sigma \, \leq \,0 \,  .  \phantom{\quad}
\end{equation*}
As it will be clear from the forthcoming analysis (see in particular Theorem~\ref{thm:main_conf} and the subsequent Remerk~\ref{rem:der2}), a sufficient condition for the rigidity statement is that the equality is achieved in the above formula, for some $p \geq 3$. 

\begin{remark}
To provide a geometric interpretation of the condition $U_p''(0)=0$, we recall from Subsection~\ref{sub:strategy} that the rigidity case in Theorem~\ref{thm:main}-(iii) corresponds to a Riemannian splitting in the conformal reformulation of the problem in terms of $g=(1-u^2)^{2/(n-2)}$ and $\ffi = \log\left[({1+u})/({1-u}) \right]$. Moreover, the splitting direction will be given by $\na \ffi$. Hence, a necessary condition for the splitting is that $\Ric_g (\nu_g, \nu_g) \equiv 0$ on $M$, where $\nu_g = \na \ffi/|\na \ffi|_g$. On the other hand, the condition $U_p''(0) =0$ can be expressed in terms of the conformally related quantities as
\begin{equation*}
\int\limits_{\pa M}
\!
  |\na \ffi |_g^{p-2} \,  \Ric_g (\nu_g,\nu_g) \, \rmd \sigma_{\!g} \,\, = \,\, 0 \, .
\end{equation*}
In other words, the rigidity statement in Theorem~\ref{thm:main}-(iii) says that if the normal component of $\Ric_g$ vanishes in average on $\pa M$, then it vanishes everywhere and $(M,g)$ splits a line.
\end{remark}

To rephrase the above discussion in a more intrinsic way, we observe that since $\pa M$ is totally geodesic and the scalar curvature $\RRR$ of $(M,g_0,u)$ is identically zero, the Gauss equation gives $2 \, \Ric(\nu, \nu) = - \RRR^{\pa M}$, where $\RRR^{\pa M}$ is the scalar curvature of the metric $g_{\pa M}$ induced by $g_0$ on the boundary of $M$. This leads to the following theorem, which is the analog of Theorem~\ref{thm:main2_geom} under null Dirichlet boundary conditions.
\begin{theorem}
\label{thm:rig_gen}
Let $(M,g_0,u)$ be an asymptotically flat solution to problem~\eqref{eq:pb_static_vacuum} in the sense of Definition~\ref{ass:AF} with $u_0=0$ and ADM mass equal to $m>0$.
Then, for every $p \geq 3$, it holds
\begin{equation}
\label{eq:cond_1}
4 \, \Big(\frac{n-1}{n-2}\Big)
\int\limits_{\pa M} \! |\DDD u|^{p}\,  \rmd\sigma \,\, \leq \,\, \int\limits_{\pa M} \! |\DDD u|^{p-2} \,\, {\RRR^{\pa M}} \,\, \rmd\sigma  ,
\end{equation}
where $\RRR^{\pa M}$ denotes the scalar curvature of the metric induced by $g_0$ on $\pa M$. Moreover, the equality holds for some $p \geq 3$ if and only if $(M,g_0,u)$ is isometric to a Schwarzschild solution with ADM mass equal to $m>0$. In particular, the boundary of $M$ has only one connected component and it is isometric to a $(n-1)$-dimensional sphere.
\end{theorem}

In the framework of overdetermined boundary value problems, we have the following corollary, which should be compared with Corollary~\ref{cor:over}.
\begin{corollary}
\label{cor:over2}
Let $(M,g_0,u)$ be an asymptotically flat solution to problem~\eqref{eq:pb_static_vacuum} in the sense of Definition~\ref{ass:AF} with $u_0=0$ and ADM mass equal to $m>0$. Assume in addition that the identity
\begin{equation}
\label{eq:over}
\left(\frac{\,2\, |\D u|\, }{n-2} \right)^{\!\!2} \,= \,\frac{\quad\RRR^{\pa M}}{(n-1)(n-2)}	
\end{equation}
holds $\mathscr{H}^{n-1}$-almost everywhere on $\pa M$. Then, the static solution $(M,g_0,u)$ is isometric to a Schwarzschild solution with ADM mass equal to $m>0$. In particular, the boundary of $M$ has only one connected component and it is isometric to a $(n-1)$-dimensional sphere.
\end{corollary}

To illustrate some other consequences of Theorem~\ref{thm:rig_gen}, we apply H\"older inequality to the right hand side of~\eqref{eq:cond_1} with conjugate exponents $p/(p-2)$ and $p/2$, obtaining
\begin{equation*}
\int\limits_{\pa M} \! |\D u|^{p-2} {\RRR^{\pa M}} \, \rmd \sigma  \,\, \leq \,\, \Bigg(  \int\limits_{\pa M} \! |\D u|^{p}  \, \rmd \sigma \Bigg)^{\!\!{(p-2)}/{p}} 
\Bigg(  \int\limits_{\pa M} \! \big|\RRR^{\pa M}\big|^{p/2}  \, \rmd \sigma \Bigg)^{\!\!{2}/{p}} \,.
\end{equation*}
This immediately implies the following corollary, which is the counterpart of Corollary~\ref{cor:Lp}.
\begin{corollary}
\label{cor:R_Lp}
Let $(M,g_0,u)$ be an asymptotically flat solution to problem~\eqref{eq:pb_static_vacuum} in the sense of Definition~\ref{ass:AF} with $u_0=0$ and ADM mass equal to $m>0$.
Then, for every $p \geq 3$, the inequality 
\begin{equation}
\label{eq:R_Lp}
  \left|\left|   \frac{\, 2 \, \D u \, }{n-2} \right|\right|_{L^p(\pa M)}  \!\! \leq \,\,\,\, \sqrt{ \,\, \left|\left|  \frac{\quad\RRR^{\pa M}}{(n-1)(n-2)} \right|\right|^{\phantom{1}}_{L^{p/2}(\pa M)} }
\end{equation}
holds true, where $\RRR^{\pa M}$ denotes the scalar curvature of the metric induced by $g_0$ on $\pa M$. Moreover, the equality holds for some $p \geq 3$ if and only if $(M,g_0,u)$ is isometric to a Schwarzschild solution with ADM mass equal to $m>0$. In particular, the boundary of $M$ has only one connected component and it is isometric to a $(n-1)$-dimensional sphere.
\end{corollary}
Letting $p \to + \infty$ in formula~\eqref{eq:R_Lp}, we obtain, under the hypotheses of the above corollary, the $L^{\infty}$-bound
\begin{equation}
\label{eq:R_Linf}
  \left|\left|   \frac{\, 2 \, \D u \, }{n-2} \right|\right|_{L^\infty(\pa M)}  \!\! \leq \,\,\,\, \sqrt{ \,\, \left|\left|  \frac{\quad\RRR^{\pa M}}{(n-1)(n-2)} \right|\right|^{\phantom{1}}_{L^{\infty}(\pa M)} } .
\end{equation}
Theorem~\ref{thm:rig_gen} gives interesting corollaries when $\pa M$ is contained in a level set of $|\D u|$ or, more geometrically, when $\pa M$ is connected. 
In fact, using Corollary~\ref{cor:R_Lp} in place of Corollary~\ref{cor:Lp}, we arrive at the following analog of Theorem~\ref{thm:ul_bounds}.
\begin{theorem}
\label{thm:ul_bounds_bound}
Let $(M,g_0,u)$ be an asymptotically flat solution to problem~\eqref{eq:pb_static_vacuum} in the sense of Definition~\ref{ass:AF} with $u_0 =0$ and ADM mass equal to $m>0$. Then, the following statements hold true.
\begin{itemize}
\item[(i)]
For every $p \geq 3$, the inequalities 
\begin{equation}
\label{eq:pi_gen_bound}
\phantom{\,\,\qquad}\frac{1}{2}  \,\, K(n,p,0) \,\left({   \frac{| \pa M| }{|\Sph^{n-1}|} } \right)^{\!\!\frac{n-2}{n-1}}\!\! \leq \,\, m \,\,  \leq  \,\, \frac{1}{2}  \, 
\sqrt{
\,\,
\left|\left|  \frac{\quad\RRR^{\pa M}}{(n-1)(n-2)} \right|\right|^{\phantom{1}}_{L^{p/2}_0( \pa M)} 
}
 \,\,\left({   \frac{| \pa M| }{|\Sph^{n-1}|} } \right)  ,
\end{equation}
hold true, where $|\pa M|$ denotes the $\mathscr{H}^{n-1}$-measure of $\pa M$ and, according to Theorem~\ref{thm:ul_bounds} we set
$$
K(n,p,0) \, = \, \left[ \frac{ || \D u||_{L_0^1( \pa M)}   }{  || \D u||_{L_0^p( \pa M)}  } \right]^{\!\frac{p(n-2)}{(p-1)(n-1)}} \!\! .
$$
Moreover, the equality is fulfilled in either the first or in the second inequality, for  some $p \geq 3$, if and only if the static solution $(M,g_0,u)$ is isometric to a Schwarzschild solution with ADM mass equal to $m>0$.

\smallskip

\item[(ii)] Assume that for some $p \geq 3$ the scalar curvature $\RRR^{\pa M}$ of the boundary satisfies the  bound
\begin{equation}
\label{eq:R_bound}
\sqrt{
\,\,
\left|\left|  \frac{\!\RRR^{\pa M}}{\,\,\RRR^{\Sph^{n-1}}} \right|\right|^{\phantom{1}}_{L^{p/2}_0( \pa M)} 
}
\,\,\, \leq \,\,\, K(n,p,0) \,\, 
\left({   \frac{|\Sph^{n-1}|}{| \pa M| } } \right)^{\!\frac{1}{n-1}} \!\! ,
\end{equation}
where $\RRR^{\Sph^{n-1}}\!\! = (n-1)(n-2)$ denotes the scalar curvature of the standard $(n-1)$-dimensional sphere,
then the static solution $(M,g_0,u)$ is isometric to a Schwarzschild solution with ADM mass equal to $m>0$.

\smallskip

\item[(iii)] Suppose that $\pa M$ is connected or, more in general, that $\pa M$ is contained in a level set of $|\D u|$. Then, the inequalities 
\begin{equation}
\label{eq:piip}
\phantom{\quad\quad} \frac12  \left({   \frac{|\pa M| }{|\Sph^{n-1}|} } \right)^{\!\!\frac{n-2}{n-1}} \! \leq \,\, m \,\, \leq \,\, \frac12  \left({   \frac{|\pa M| }{|\Sph^{n-1}|} } \right)^{\!\!\frac{n-2}{n-1}}  \! \sqrt{  \left({   \frac{|\pa M| }{|\Sph^{n-1}|} } \right)^{\!\!-\frac{n-3}{n-1}} \!\!
\frac{\int_{\pa M} \!  {\RRR^{\pa M}} \,  \rmd\sigma}%{ \int\limits_{\Sph^{n-1}}  \!\!{\RRR^{\Sph^{n-1} }} \rmd\sigma_{\Sph^{n-1}}    }
{(n-1) (n-2) |\Sph^{n-1}|} 
} \,\,
\end{equation}
hold true. Moreover, the equality holds in either the first or in the second inequality in the above formula, for  some $p \geq 3$, if and only if the static solution $(M,g_0,u)$ is isometric to a Schwarzschild solution with ADM mass equal to $m>0$. 

\smallskip

\item[(iv)] Under the same assumptions as in {\rm (iii)}, if the scalar curvature $\RRR^{\pa M}$ of the boundary satisfies the bound
\begin{equation}
\label{eq:scal_invrad}
\sqrt{ \bigg|\frac{\!\RRR^{\pa M}}{\,\,\RRR^{\Sph^{n-1}}}\bigg|} \,\, \leq \,\, \left({   \frac{|\Sph^{n-1}|}{|\pa M| } } \right)^{\!\!\frac{1}{n-1}} \!\! ,
\end{equation}
%where $\RRR^{\Sph^{n-1}}\!\! = (n-1)(n-2)$ denotes the scalar curvature of the standard $(n-1)$-dimensional sphere, 
then $(M,g_0,u)$ is isometric to a Schwarzschild solution with ADM mass equal to $m>0$. 
\end{itemize}
\end{theorem}
\begin{remark}
The first inequality in~\eqref{eq:piip} is the well known Riemannian Penrose Inequality, which is known to hold up to dimension $n=7$ on asymptotically flat manifolds with nonnegative scalar curvature and compact minimal boundary. For a comprehensive discussion about the general Riemannian Penrose Inequality and its generalizations up to dimension $n=7$ we refer the reader to~\cite{Hui_Ilm, Bra, Bra_Lee} and the references therein.
\end{remark}
\begin{proof}
To prove (i) it is sufficient to mimic the proof of Theorem~\ref{thm:ul_bounds}, working on $\pa M= \{ u=0\}$ and using Corollary~\ref{cor:R_Lp} in place of Corollary~\ref{cor:Lp}.
Statement (ii) follows immediately from statement (i), since  inequality~\eqref{eq:R_bound} implies that the equality is satisfied in~\eqref{eq:pi_gen_bound}. To obtain the first inequality in (iii), we observe that when $\pa M$ is included in a level set of $|\D u|$, say $\{ |\D u| = c \}$, the constant $K(n,p,0)$ in the first inequality of~\eqref{eq:pi_gen_bound} is identically equal to $1$. To obtain the second inequality in (iii), it is sufficient to observe that, by Theorem~\ref{thm:rig_gen} one has
\begin{equation*}
%4 \, \Big(\frac{n-1}{n-2}\Big) \, \frac{1}{|\pa M|} \,\, \bigg|  \int\limits_{\pa M} \! |\DDD u|\,  \rmd\sigma  \, \bigg|^2 \,\, \leq \,\,
4 \, \Big(\frac{n-1}{n-2}\Big)  \, |\pa M|\, c^2
%\int\limits_{\pa M} \! |\DDD u|^{2}\,  \rmd\sigma 
\,\, \leq \,\, \int\limits_{\pa M} \!  {\RRR^{\pa M}} \,\, \rmd\sigma \, ,
\end{equation*}
%where the first inequality follows by Jensen's inequality, whereas the second one follows from Theorem~\ref{thm:rig_gen} and from the (global) constancy of $|\D u|$ on $\pa M$. 
On the other hand, the left hand side of the above formula can be rewritten with the help of~\eqref{eq:U1_mass} as 
$$
(2m)^2 \,\, \bigg(\frac{|\Sph^{n-1}|}{|\pa M|}\bigg) \, (n-1) (n-2) |\Sph^{n-1}|\,.
$$
The second inequality in~\eqref{eq:piip} is now a consequence of simple algebraic manipulations, whereas the rigidity statement follows easily from Theorem~\ref{thm:rig_gen}. Finally, we observe that (iv) follows immediately from (ii), noticing that condition~\eqref{eq:scal_invrad} implies that the term under square root in~\eqref{eq:piip} is less than or equal to $1$.
\end{proof}

To describe some immediate consequences of the above theorem, we observe that, in dimension $n=3$ and when $\pa M$ is connected, the Gauss-Bonnet Formula gives 
\begin{equation*}
{\int\limits_{\pa M} \!  {\RRR^{\pa M}}  \rmd\sigma} \,
= \, 4\pi \, \chi(\pa M) \, \leq \, 8  \pi \, ,
\end{equation*}
where $\chi(\pa M)$ is the Euler characteristic of $\pa M$.
In particular, the term under square root in~\eqref{eq:piip} is always bounded above by $1$. Hence the equality holds in~\eqref{eq:piip} and we can recover the classical $3$-dimensional Black Hole Uniqueness Theorem.
\begin{theorem}[Black Hole Uniqueness Theorem]
\label{thm:bhu}
Let $(M,g_0,u)$ be a $3$-dimensional asymptotically flat solution to problem~\eqref{eq:pb_static_vacuum} in the sense of Definition~\ref{ass:AF} with $u_0=0$ and ADM mass equal to $m>0$. Moreover, suppose that $\pa M$ is connected. Then, $(M,g_0,u)$ is isometric to a Schwarzschild solution with ADM mass equal to $m>0$.
\end{theorem}

Coming back to formula~\eqref{eq:piip}, it is important to notice that the term under the square root is scaling invariant. In fact, it can be rewritten in terms of the $(n-1)$-dimensional renormalized Einstein-Hilbert functional. We recall that for a compact $(n-1)$-dimensional manifold $\Sigma$, this functional is defined as 
\begin{equation}
g \, \longmapsto \, \mathscr{E}^\Sigma_{n-1}(g) \, = \, {|\Sigma|^{- \frac{n-3}{n-1}}_g} \!\! \int\limits_{\Sigma}  \RRR_g \, \rmd \sigma_{\!g} \,,
\end{equation}
where $|\Sigma|_g$ represents the $(n-1)$-dimensional volume of $\Sigma$ computed with respect to the metric $g$, whereas $\rmd \sigma_{\! g}$ and $\RRR_g$ are respectively the volume element and the scalar curvature of $g$. The minimizers of the renormalized Einstein-Hilbert functional over a given conformal class are constant scalar curvature metrics called Yamabe metrics. It follows from the celebrated works of Aubin and Schoen on the resolution of the Yamabe problem that for every compact $(n-1)$-dimensional manifold $\Sigma$, with $n\geq 4$, it holds
\begin{equation*}
\sup \left\{ \mathscr{E}^\Sigma_{n-1}(g) \, \left|  \, \hbox{$g$ is a Yamabe metric on $\Sigma$} \right.  \right\} \,\, \leq \,\, \mathscr{E}_{n-1}^{\Sph^{n-1}} ({g}_{\Sph^{n-1}}) \, .
\end{equation*}
In this setting, formula~\eqref{eq:piip} can be rephrased as
\begin{equation}
\label{eq:piip2}
\frac12  \left({   \frac{|\pa M| }{|\Sph^{n-1}|} } \right)^{\!\!\frac{n-2}{n-1}} \! \leq \,\, m \,\, \leq \,\, \frac12  \left({   \frac{|\pa M| }{|\Sph^{n-1}|} } \right)^{\!\!\frac{n-2}{n-1}}  \!\!\! \sqrt{ \frac{\mathscr{E}_{n-1}^{\pa M} ({g}_{\pa M}) }{ \mathscr{E}_{n-1}^{\Sph^{n-1}} ({g}_{\Sph^{n-1}})}
} \, .
\end{equation}
This gives the following theorem, which shows how the rotational symmetry of the static solution $(M,g_0,u)$ can be detected from the knowledge of the intrinsic geometry of the boundary, in dimension $n\geq 4$.
\begin{theorem}
\label{thm:bhun}
For every $n\geq 4$, let $(M,g_0,u)$ be a $n$-dimensional asymptotically flat solution to problem~\eqref{eq:pb_static_vacuum} in the sense of Definition~\ref{ass:AF} with $u_0=0$ and ADM mass equal to $m>0$. Moreover, suppose that $\pa M$ is connected. Then, we have
\begin{equation}
\label{eq:yam}
{ \mathscr{E}_{n-1}^{\Sph^{n-1}} ({g}_{\Sph^{n-1}})}
\, \leq \, {\mathscr{E}_{n-1}^{\pa M} ({g}_{\pa M}) } \, ,
\end{equation}
where $g_{\pa M}$ is the metric induced by $g_0$ on $\pa M$.
Moreover, the equality holds if and only if $(M,g_0,u)$ is isometric to a Schwarzschild solution with ADM mass equal to $m>0$. In particular, if $g_{\pa M}$ is a Yamabe metric, then 
$(M,g_0,u)$ is rotationally symmetric. 
\end{theorem}

\subsection{Willmore-type inequalities} 
\label{sub:will}
It is well known that the mean curvature of a smooth closed hypersurfaces $\Sigma$ embedded in the flat $n$-dimensional Euclidean space satisfies the Willmore inequality
\begin{equation*}
|\Sph^{n-1}|^{\frac{1}{n-1}} \,\, \leq \,\, \left|\left|   \frac{\HHH}{n-1} \right|\right|_{L^{n-1}(\Sigma)}  ,
\end{equation*}
where $\HHH$ is the mean curvature of $\Sigma$. Moreover, the equality holds if and only if $\Sigma$ is isometric to a round sphere. As a consequence of our analysis, we are able to prove analogous inequalities for the level sets of $u$ in the setting of problem~\eqref{eq:pb_static_vacuum}, provided $n\geq 4$.
\begin{theorem}[Willmore-type Inequalities]
\label{thm:will}
For every $n\geq 4$, let $(M,g_0,u)$ be a $n$-dimensional asymptotically flat solution to problem~\eqref{eq:pb_static_vacuum} in the sense of Definition~\ref{ass:AF} with $0 \leq u_0<1$ and ADM mass equal to $m>0$. Then, the following statements hold true.
\begin{itemize}
\item[(i)] For every $t \in [u_0, 1) \cap (0,1)$, the inequality
\begin{equation}
\label{eq:will_t}
|\Sph^{n-1}|^{\frac{1}{n-1}} \,\, \leq \,\, \frac{1}{t} \,\left|\left|   \frac{\HHH}{n-1} \right|\right|_{L^{n-1}(\{u=t\})} \phantom{00000000}
\end{equation}
holds true. Moreover, the equality is fulfilled for some $t \in [u_0, 1) \cap (0,1)$ if and only if the level set $\{u=t\}$ is isometric to a round sphere and the static solution $(M,g_0,u)$ is isometric to a Schwarzschild solution with ADM mass equal to $m>0$. 
\smallskip
\item[(ii)] Assume that $u_0=0$. Then, the inequality
\begin{equation}
\label{eq:will_0}
|\Sph^{n-1}|^{\frac{1}{n-1}} \,\, \leq  \,\,\, \sqrt{ \,\, \left|\left|  \frac{\quad\RRR^{\pa M}}{(n-1)(n-2)} \right|\right|^{\phantom{1}}_{L^{\frac{n-1}{2}}(\pa M)} }
\end{equation}
holds true. Moreover, the equality is fulfilled if and only if $\pa M$ is isometric to a round sphere and the static solution $(M,g_0,u)$ is isometric to a Schwarzschild solution with ADM mass equal to $m>0$. 
\end{itemize}
\end{theorem}
\begin{proof}
To prove the inequality in (i), we observe that combing formula~\eqref{eq:L1_Lp} with identity~\eqref{eq:U1_mass} and inequality~\eqref{eq:Lp_ineq} in Corollary~\ref{cor:Lp}, one gets
\begin{align*}
\left(\frac{|\Sph^{n-1}|}{|\{ u=t \}|} \right)^{\!\frac{1}{p}}  \leq \,\, \frac{1}{2m} \, \Big(\frac{2m}{1-t^2}\Big)^{\!\!\!\frac{(p-1) \, (n-1)}{p \, (n-2)}} \,  \left|\left|  \frac{\, 2 \, \D u \, }{n-2} \right|\right|_{L_0^p(\{ u=t \})} \leq \,\, \frac{1}{t} \,  \Big(\frac{2m}{1-t^2}\Big)^{\! \frac{(p-1) \, (n-1)}{p \, (n-2)}-1} \left|\left|  \frac{\HHH}{n-1} \right|\right|_{L^p_0(\{ u=t \})} \!\!\!, 
\end{align*}
for every $p\geq 3$. If $n\geq 4$, we can then choose $p= n-1$ in the above formula, obtaining 
\begin{equation*}
\left(\frac{|\Sph^{n-1}|}{|\{ u=t \}|} \right)^{\!\frac{1}{n-1}} \,\,\leq \,\, \frac{1}{t} \,  \left|\left|  \frac{\HHH}{n-1} \right|\right|_{L^{n-1}_0(\{ u=t \})} \!\!\!.
\end{equation*}
Simplifying the above expression, we arrive at the desired inequality~\eqref{eq:will_t}. The rigidity statement in the equality case is now an easy consequence of Corollary~\ref{cor:Lp}.

To prove the inequality in (ii) we observe that, combining formula~\eqref{eq:L1_Lp} with identity~\eqref{eq:U1_mass} and  letting $t=0$, one gets
\begin{equation*}
\left(\frac{|\Sph^{n-1}|}{|\pa M|} \right)^{\!\frac{1}{p}}  \leq \,\,  ({2m})^{ \frac{(p-1) \, (n-1)}{p \, (n-2)}-1} \,  \left|\left|  \frac{\, 2 \, \D u \, }{n-2} \right|\right|_{L_0^p(\pa M)} ,
\end{equation*}
for every $p\geq 3$. Using~\eqref{eq:R_Lp} in Corollary~\ref{cor:R_Lp} and setting $p= n-1$, with $n\geq 4$, in the above formula, we obtain
\begin{equation*}
\left(\frac{|\Sph^{n-1}|}{|\pa M|} \right)^{\!\frac{1}{n-1}}  \leq  \,\,\, \sqrt{ \,\, \left|\left|  \frac{\quad\RRR^{\pa M}}{(n-1)(n-2)} \right|\right|^{\phantom{1}}_{L_0^{\frac{n-1}{2}}(\pa M)} } ,
\end{equation*}
which is the desired inequality~\eqref{eq:R_Lp}, up to a simple normalization. The rigidity statement in the equality case follows at once from Corollary~\ref{cor:R_Lp}.
\end{proof}

%%%%%%%%%%%%%%%%%%%%%%%%%%%%%%%%%
%%%%%%%%%%%%%%%%%%%%%%%%%%%%%%%%%

\section{A conformally equivalent formulation of the problem}
\label{sec:conf_reform}

%%%%%%%%%%%%%%%%%%%%%%%%%%%%%%%%%
%%%%%%%%%%%%%%%%%%%%%%%%%%%%%%%%%

\subsection{A conformal change of metric.}
\label{sub:conf}

The aim of this section is to reformulate system~\eqref{eq:pb_static_vacuum} in a conformally equivalent setting. First of all, we notice that if $(M,\go, u)$ is an {\em asymptotically flat static solution} in the sense of Definition~\ref{ass:AF}, then one has that $1-u^2 > 0$ everywhere in $M$, by the Strong Maximum Principle.
Motivated by the explicit formul\ae~\eqref{eq:sol_schwarz} of the Schwarzschild solution, we are led to consider the following conformal change of metric
\begin{equation}
\label{eq:g_schw}
g
\,=\,
({1-u^2})^{\frac2{n-2}} \, \go \, .
\end{equation}
It is immediately seen that when $u$ and $\go$ are as in~\eqref{eq:sol_schwarz} then $g$ is a cylindrical metric. Hence, we will refer to the conformal change~\eqref{eq:g_schw} as to a {\em cylindrical ansatz}. In any case, if $(M,g_0,u)$ is {\em asymptotically flat}, it is not hard to deduce from the expansions~\eqref{eq:AF} and~\eqref{eq:uexp} that the metric $g$ is {\em asymptotically cylindrical}. In fact, with the notations introduced in Definition~\ref{ass:AF}, we have that in $M \setminus K$ the metric $g$ satisfies the expansion
\begin{align}
\label{ass:AC}
\nonumber g \, = \, ({1-u^2})^{\frac2{n-2}} \, g_0 \,& = \, \big(  \, 1 - (1 -m|x|^{2-n} + w )^2   \, \big)^{\frac2{n-2}} \,( \delta_{\alpha \beta}+\eta_{\alpha \beta} ) \,\cdot \, dx^\alpha \!\otimes dx^\beta\\
\nonumber & = \, (2m)^{\frac{2}{n-2}} \,  \big[ \, (1 + o_2(1)) \,  |x|^{-2}{\delta_{\alpha \beta}}  +   \, (1 + o_2(1)) \,  |x|^{-2}{\eta_{\alpha \beta}}   \, \big] \,\cdot \, dx^\alpha \!\otimes dx^\beta  \\
&= \, (2m)^{\frac{2}{n-2}} \, (1 + o_2(1)) \, \left[ \,{d|x|\otimes d|x|}  \, + \,  g_{\Sph^{n-1}}  \, \right] \,  +   \, 
\sigma_{\alpha\beta}\,\cdot \, dx^\alpha \!\otimes dx^\beta  ,
\end{align}
with $\sigma_{\alpha \beta} = o_2 \big(|x|^{-(n+2)/2}\big)$, as $|x| \to +\infty$. To describe how this fact will be exploited in the proof of our results, let us first observe that another straightforward implication of the expansions~\eqref{eq:AF} and~\eqref{eq:uexp} is that there exists $0< t_0 <1$ such that for every $ t_0 \leq t < 1$ the level set $\{ u= t \}$ is regular (meaning that $|\D u|>0$ on the level set) and diffeomorphic to a $(n-1)$-dimensional sphere. In particular, the manifold with boundary $\{ u\geq t_0 \}$ is diffeomeorphic to the cylinder $[t_0 , 1) \times \{ u = t_0\}$. Hence, it is possible to choose in this region a local system of coordinates $\{ u,\vartheta^1,\!....,\vartheta^{n-1}\}$, where $\{\vartheta^1,\!...., \vartheta^{n-1} \}$ are local coordinates on $\{ u = t_0 \}$. In such a system, the metric $\go$ can be written as
\begin{equation}
\label{eq:expr_go}
\go \, = \,\frac{du \otimes du}{|\D u|^2} +\cgo_{ij}(u,\vartheta^1\!,\!..., \vartheta^{n-1})\,d\vartheta^i\!\otimes d\vartheta^j ,
\end{equation}
where the latin indices vary between $1$ and $n-1$. On the other hand, it is not hard to check that the expansions~\eqref{eq:AF} and~\eqref{eq:uexp} also imply 
\begin{align*}
\frac{\D u}{|\D u|} & \, = \,  \left[ \frac{\,\,\, x^\alpha}{|x|} + o_1(1) \right] \, \frac{\pa}{ \pa x^\alpha} \, , \quad \hbox{as $|x| \to + \infty$}\,.
\end{align*}
This means that the level sets of $u$ tend to coincide with the level sets of $|x|$, as $|x| \to + \infty$ or, equivalently, as $u \to 1$. In light of these remarks, it follows from~\eqref{ass:AC} and~\eqref{eq:expr_go} that 
\begin{equation}
\label{eq:ACg}
g_{ij}(u,\vartheta^1\!,\! ... , \vartheta^{n-1}) \cdot d\vartheta^i\!\otimes d\vartheta^j \, = \, (1-u^2)^{\frac{2}{n-2}} \, \cgo_{ij}(u,\vartheta^1\!,\! ... , \vartheta^{n-1}) \cdot d\vartheta^i\!\otimes d\vartheta^j  \, \longrightarrow \, (2m)^{\frac{2}{n-2}} \, \g_{\Sph^{n-1}}  \,, 
%\qquad \hbox{as $u \to 1$} \, , 
\end{equation}
as $u \to 1$. In particular, as it is natural to expect for an {\em asymptotically cylindrical} metric, we have that the $g$-hypersurface area functional for the level sets of $u$ is uniformly bounded at infinity, namely
\begin{equation}
\label{eq:uareagrowth}
\sup_{ t_0 \leq t < 1} \!\! \int\limits_{ \{ u = t\}} \!\!\!\rmd \sigma_{\!\g} \, < \, + \infty \, ,
\end{equation}
where $\rmd \sigma_{\!\g}$ denotes the volume element of the metric induced by $g$ on the level sets. 
\begin{remark}
\label{rem:pr_round}
From this preliminary discussion it is clear that in the case where the metric $g$ has a product structure, with the level sets of $u$ as cross sections, then the coefficients $g_{ij}$'s in formula~\eqref{eq:ACg} do not depend on the variable $u$. This implies in turn that the metric $g$ is everywhere rotationally symmetric.
\end{remark}

Our next task is to reformulate the problem~\ref{eq:pb_static_vacuum} in terms of the metric $g$.
To this aim we fix local coordinates $\{y^{\alpha}\}_{\alpha=1}^n$ in $M$
and using standard formul\ae\ for conformal changes of metrics, 
we deduce that the Christoffel symbols $\Gamma_{\a\b}^\gamma$ and $\GGG_{\a\b}^\gamma$, of the metric $g$ and
$\go$ respectively, are related to each other via the identity
\begin{equation*}
\Gamma_{\alpha\beta}^{\gamma}
\,=\,\,
\GGG_{\alpha\beta}^{\gamma}
\, - \, \frac{2u}{(n-2)(1-u^2)} \, \Big(\, 
\delta_{\alpha}^{\gamma}\,\pa_{\beta}u
+\delta_{\beta}^{\gamma}\,\pa_{\alpha}u
-\cgo_{\alpha\beta}\,\ccgo^{\gamma\eta}\,\pa_{\eta}u \, \Big)\, .                                                
\end{equation*}
Comparing the local expressions for the Hessians of a given function $w \in {\mathscr C}^2(M)$ with respect to the metrics $g$ and $\go$, namely
$\na^2_{\alpha\beta}w=\pa^{\,2}_{\alpha\beta}w-
\Gamma_{\alpha\beta}^{\gamma}\pa_{\gamma}w$ and $\DD_{\alpha\beta}w=\pa^{\,2}_{\alpha\beta}w-
\GGG_{\alpha\beta}^{\gamma}\pa_{\gamma}w$,
one gets 
\begin{align*}
\na^2_{\alpha\beta}w
     &\,=\,\DD_{\alpha\beta}w \, + \, \frac{2u}{(n-2)(1-u^2)}
            \, \Big(\, \pa_{\alpha}w\,\pa_{\beta}u
                 +\pa_{\beta}w\,\pa_{\alpha}u
              -\langle\D w\, | \, \D u\rangle\,
              \cgo_{\alpha\beta}\, \Big)\,, \\
\Deg w &\,=\,(1-u^2)^{-\frac2{n-2}}
\Big( \, \De w-\frac{2u}{1-u^2} \, \langle\D w \,| \,\D u\rangle   \, \Big) \, .                                
\end{align*}
We note that in the above expressions as well as in the following ones, the notations $\na$ and $\Delta_g$ represent the Levi-Cita connection and the Laplace-Beltrami operator of the metric $g$. In particular, letting $w=u$ and using $\De u = 0$, one has
\begin{align}
\label{eq:hess_change_1}
\na^2_{\alpha\beta}u
     &\,=\,\DD_{\alpha\beta}u \, + \, \frac{2u}{(n-2)(1-u^2)}
            \,\Big( \, 2 \,\pa_{\alpha}u\,\pa_{\beta}u
             \, - \, |\D u|^2\,
              \cgo_{\alpha\beta} \, \Big) \, ,\\         
\label{eq:lapl_change_1}
\Deg u
     &\,=\,  - \, \frac{2u  }{(1-u^2)^{\frac{n}{n-2}}} \,\, |\D u |^2 
\, .                       
\end{align}
To continue, we observe that the Ricci tensor $\Ricg=\cRicg_{\alpha\beta}\,dy^{\alpha}\!\otimes dy^{\beta}$ of the metric $\g$ can be expressed 
\vspace{-0.4cm}\\
in terms of the Ricci tensor $\Ric = \cRic_{\a\b}\,dy^{\alpha}\!\otimes dy^{\beta}$ of the metric $\go$ as 
\begin{equation*}
\cRicg_{\a\b} \,\, = \,\, \cRic_{\a\b} \,  + \, \frac{2u}{1-u^2} \, \DD_{\a\b} u
\, + \, \Big(\frac{2}{n-2}\Big)  
\frac{n-2+ n u^2}{ (1-u^2)^2} \, \pa_\a u \, \pa_\b u  \, + \, \Big(\frac{2}{n-2}\Big)  \frac{|\D u|^2}{1-u^2} \, \cgo_{\a\b}  \, ,
\end{equation*}
where we have used the fact that $\De u = 0$. If in addition we plug the equation $u \,\Ric = \DD u$ in the above formula, we obtain
\begin{equation}
\label{eq:ric_change_1}
\cRicg_{\a\b} \,\, = \,\,  \frac{1+u^2}{u  (1-u^2)} \, \DD_{\a\b} u
\, + \, \Big(\frac{2}{n-2}\Big)  
\frac{n-2+ n u^2}{ (1-u^2)^2} \, \pa_\a u \, \pa_\b u  \, + \, \Big(\frac{2}{n-2}\Big)  \frac{|\D u|^2}{1-u^2} \, \cgo_{\a\b}  \, .
\end{equation}
To obtain nicer formul\ae, it is convenient to introduce the new variable 
\begin{equation}
\label{eq:newvariable}
\ffi \, =\, \log\Big(\frac{1+u}{1-u} \Big)
\qquad\Longleftrightarrow\qquad
u\,=\,\tanh\Big(\frac{\ffi}2\Big)\,.
\end{equation}
As a consequence we have that 
\begin{align}
\label{eq:defideu}
\pa_\a \ffi & \, = \, \frac{2 }{1-u^2}  \, \pa_\a u \, , \\
\label{eq:dedefidedeu}
\nana_{\a\b} \ffi & \, = \,  \frac{2 }{1-u^2} \, \DD_{\a\b} u \, + \, \Big(\frac{n}{n-2}\Big) \, \frac{4u}{(1-u^2)^2} \,  \Big(    \pa_\a u \, \pa_\b u  \,  - \, \frac{ \,\, |\D u|^2}{n} \cgo_{\a\b}  \, \Big)\, .
\end{align}
For future convenience, we report the relation between $|\na \ffi|^2_\g$ and $|\D u|^2$ as well as the one between $|\nana \ffi|_\g^2$ and $|\DD u|^2$, namely
\begin{align*}
|\na \ffi|^2_\g &\, = \, 4 \,\,  \frac{ |\D u|^2 \phantom{|\D|}}{(1-u^2)^{2\frac{n-1}{n-2}}} \, , \\
|\nana \ffi |_\g^2 
& \, = \, 4 \, \frac{|\DD u|^2 }{(1-u^2)^{\frac{2n}{n-2}}} \, + \, \frac{16  n }{n-2} \,  \frac{  u \, \DD u  (\D u, \D u)}{(1-u^2)^{\frac{3n-2}{n-2}}} \, + \, \frac{16n(n-1)}{(n-2)^2} \, \frac{  u^2 \,|\D u|^4}{(1-u^2)^{\frac{4n-4}{n-2}}} \, .
\end{align*}
On the other hand, in virtue of the expansions~\eqref{eq:AF} and~\eqref{eq:uexp}, it is immediate to deduce that the quantities 
\begin{equation*}
\frac{ |\D u| \phantom{|\D|}}{(1-u^2)^{\frac{n-1}{n-2}}} \qquad \hbox{and} \qquad \frac{|\DD u| }{(1-u^2)^{\frac{n}{n-2}}}
\end{equation*}
are uniformly bounded in $M$. Taking this fact into account, one can easily estimate the above expressions for $|\na \ffi|_g$ and $|\nana \ffi|_g$, obtaining the bound
\begin{equation*}
%\label{eq:conf_bound}
\sup_M \big(|\na \ffi|_\g + |\nana \ffi|_\g \big)\, < \, +\infty \, .
\end{equation*}
This fact will be used in the following discussion in combination with the uniform bound on the $g$- hypersurface area functional of the level sets~\eqref{eq:uareagrowth}. For the ease of reference, we summarize these estimates in the following lemma.
\begin{lemma}
\label{le:bound}
Let $(M,g_0,u)$ be an asymptotically flat static solution to problem~\eqref{eq:pb_static_vacuum} and let $g$ and $\ffi$ be the metric and the smooth function defined in~\eqref{eq:g_schw} and ~\eqref{eq:newvariable}, respectively. Then, there exist $0 \leq s_0 < + \infty$ and a positive constant $0<C<+\infty$ such that 
\begin{equation}
\label{eq:cylbounds}
%\sup_M \big(|\na \ffi|_\g + |\nana \ffi|_\g \big) 
\sup_M |\na \ffi|_\g \, + \, \sup_M |\nana \ffi|_\g
\, + \,  \sup_{ s_0 \, \leq \, s } \!\! \int\limits_{ \{ \ffi = s \}} \!\!\!\rmd \sigma_{\!\g} \,\, \leq \,\, C \, .
\end{equation}
\end{lemma}

Combining expressions~\eqref{eq:hess_change_1}, \eqref{eq:lapl_change_1}, \eqref{eq:ric_change_1} together with~\eqref{eq:defideu} and ~\eqref{eq:dedefidedeu}, we are now in the position to reformulate problem~\eqref{eq:pb_static_vacuum} as
\begin{equation}
\label{eq:pb_schw_reform}
\left\{
\begin{array}{rcll}
\displaystyle
\Ricg \, - \, \coth (\ffi) \nana\ffi \, + \, \frac{d\ffi\otimes d\ffi}{n-2}
\!\!\!\!& \, = \, &\displaystyle\!\!\!\!
\frac{|\na\ffi|^2_g}{n-2}\,\g & {\rm in }\quad M,\\
\displaystyle
\phantom{\frac12}\Deg \, \ffi\!\!\!\!& \, = \, &\!\!\!\!0 & {\rm in }\quad M,\\
\displaystyle
 \phantom{\frac12}\ffi\!\!\!\!& \, = \, &\!\!\!\! \ffi_0 &{\rm on }\ \ \pa M,\\
\displaystyle
\phantom{\frac12}\ffi(x)\!\!\!\!&\to&\!\!\!\!+\infty & 
\mbox{as }\ |x|\to +\infty \, ,
\end{array}
\right.
\end{equation}
where, according to~\eqref{eq:newvariable}, we have set $\ffi_0 =  \log\left[({1+u_0})/({1-u_0}) \right]$.
Moreover, the solution $(M,g,\ffi)$ is such that $(M,g)$ is {\em asymptotically cylindrical} in the sense described at the beginning of this section and in particular the estimate~\eqref{eq:cylbounds} is in force. In other terms, the {\em asymptotically flat static solution} $(M,g_0,u)$ corresponds via the conformal change
\begin{equation*}
g \, = \, \big(\!\cosh({\ffi}/{2})\big)^{-\frac{4}{n-2}} g_0
\end{equation*}
to an {\em asymptotically cylindrical} quasi Einstein type manifold $(M,g,\ffi)$ with unbounded $g$-harmonic potential function $\ffi$.

To describe the idea that will lead us throughout the analysis of system~\eqref{eq:pb_schw_reform}, we note that taking the trace of the first equation one gets
\begin{equation}
\label{eq:tilde_R}
\frac{\Rg}{n-1} \, = \, \frac{|\na\ffi|^2_g}{n-2} \, ,
\end{equation}
where $\Rg$ is the scalar curvature of the metric $\g$. It is important to observe that in the cylindrical situation, which is the conformal counterpart of the Schwarzschild solution, $\Rg$ has to be constant. In this case, the above formula implies that also $|\na \ffi|_g$ has to be constant. Plugging this information into the Bochner formula, it is then immediate to conclude that $\ffi$ has to be an affine function for the metric $\g$. For these reasons, also in the  situation, where we do not know a priori if $g$ is cylindrical, it is natural to think of $\na \ffi$ as to a candidate splitting direction and to investigate under which conditions this is actually the case.

\subsection{The geometry of the level sets of $\ffi$.}

In the forthcoming analysis a crucial role is be played by the study the geometry of the level sets of $\ffi$, which coincide with the level sets of $u$, by definition. 
Hence, we pass now to describe the second fundamental form and the mean curvature of the regular level sets of $\ffi$ (or equivalently of $u$) in both the original Riemannian context $(M, \go)$ and the conformally related one $(M, \g)$.
To this aim, we fix a regular level set $\{ \ffi = s_0\}$ of $\ffi$ and we note that it must be compact, by the properness of $\ffi$. In particular, there exists a real number $\delta>0$  such that in the tubular neighborhood $\mathcal{U}_\delta = \{s_0 - \delta < \ffi < s_0 + \delta \}$ we have $|\na \ffi |_\g > 0$ so that $\mathcal{U}_\delta$ is foliated by regular level sets of $\ffi$. As a consequence, $\mathcal{U}_\delta$ is diffeomorphic to $(s_0 -\delta , s_0 + \delta) \times \{ \ffi = s_0 \}$ and the function $\ffi$ can be regarded as a coordinate in $\mathcal{U}_\delta$. Thus, one can choose a local system of coordinates $\{\ffi,\vartheta^1\!,\!....,\vartheta^{n-1}\}$, where $\{\vartheta^1\!,\!...., \vartheta^{n-1} \}$ are local coordinates on $\{ \ffi = s_0 \}$. In such a system, the metric $\g$ can be written as
\begin{equation*}
\g \, = \,\frac{d\ffi \otimes d\ffi}{|\na \ffi|_\g^2} +\g_{ij}(\ffi,\vartheta^1 \!,\!...., \vartheta^{n-1})\,d\vartheta^i\!\otimes d\vartheta^j \,,
\end{equation*}
where the latin indices vary between $1$ and $n-1$. A similar expression has been obtained in~\eqref{eq:expr_go} for the metric $\go$ in terms of the local coordinates $\{u, \vartheta^1\!,\!...., \vartheta^{n-1} \}$. We now fix in $\mathcal{U}_\delta$ the $\go$-unit vector field 
$\nu =\D u/|\D u|=\D\ffi/|\D\ffi|$ and the $\g$-unit vector field $\nu_g =\na u/|\na u|_{\g}=\na\ffi/|\na\ffi|_g$. Accordingly, the second fundamental forms of the regular level sets of $u$ or $\ffi$ with respect to ambient metric $\go$ and the conformally-related ambient metric $g$ are respectively given by
\begin{equation*}
\cho_{ij}=\frac{\DD_{ij} u}{|\D u|}=\frac{\DD_{ij}\ffi}{|\D\ffi|}
\qquad\mbox{and}\qquad
\chg_{ij}=\frac{\nana_{ij} u}{|\na u|_\g}
=\frac{\nana_{ij}\ffi}{|\na\ffi|_\g}\, ,  \qquad \mbox{for}\quad i,j = 1,\!...., n-1.
\end{equation*}  
Taking the traces of the above expressions with respect to the induced metrics and using the fact that $u$ is $\go$-harmonic and $\ffi$ is $\g$-harmonic, 
we obtain the following expressions for the mean curvatures in  the two ambients
\begin{equation}
\label{eq:formula_curvature}
\Ho=-\frac{\DD u(\D u,\D u)}{|\D u|^3}\,,
\qquad\qquad
\Hg=-\frac{\nana\ffi(\na\ffi,\na\ffi)}{|\na\ffi|_{\g}^3}\,.
\end{equation}
Tacking into account  expressions~\eqref{eq:defideu} and~\eqref{eq:dedefidedeu}, one can show that 
the second fundamental forms are related by
\begin{align}
\label{eq:formula_h_h_g}
\chg_{ij}
&\,=\,
(1-u^2)^{\frac1{n-2}}
\bigg[ \, \cho_{ij}-   \Big( \frac{1}{n-2}    \Big)  
 \, \frac{2 u \, |\D u|  }{1-u^2} \,  \cgo_{ij}   \bigg] \, .
\end{align}
The analogous formula for the mean curvatures reads
\begin{align}
\label{eq:formula_H_H_g}
{\HHH_g}
&\,=\,
(1-u^2)^{-\frac1{n-2}}
\bigg[\, {\Ho} - \Big( \frac{n-1}{n-2}    \Big)  
\, \frac{2u \,|\D u|  }{1-u^2}  \, \bigg] \, .
\end{align}
For the sake of completeness, we also report the reversed formul\ae\, for the second fundamental forms
\begin{align*}
\cho_{ij}
&\,=\,
\big(\!  \cosh(\ffi/2)\big)^{\frac{2}{n-2}}
\bigg[ \, \chg_{ij} 
+ \Big(\frac{1}{n-2}\Big) \, \tanh(\ffi/2) \, |\na \ffi|_\g \, \g_{ij} \, \bigg]
%\\
%&\,=\, \big( \! \cosh(\ffi/2)\big)^{\frac{2}{n-2}}
%\bigg[ \, \chg_{ij}  +  \left\langle  \na \log\big( \! \cosh(\ffi/2)\big)^{\frac{2}{n-2}}  \, \Big| \, \nu_\g \right\rangle_{\!\g} \,  \g_{ij} \,  \bigg] \,,
\end{align*}
as well as for the mean curvatures
\begin{align*}
{\Ho}
&\,=\,
\big( \! \cosh(\ffi/2)\big)^{-\frac{2}{n-2}}
\bigg[\, {\Hg}   + \Big(\frac{n-1}{n-2}\Big)\, \tanh(\ffi/2) \, |\na \ffi|_\g  \,\bigg] \, .
\end{align*}

Concerning the nonregular level sets of $\ffi$, we first observe that by the results in~\cite{Hardt_Simon} and~\cite{Lin}, one has that the $(n-1)$-dimensional Hausdorff measure of the level sets of $\ffi$ is locally finite. Hence, the properness of $\ffi$ forces the level sets to have finite $(n-1)$-dimensional Hausdorff measure. To go further in the description of the nonregular level sets of $\ffi$, we set ${\rm Crit}(\ffi) = \{  x\in M \, | \, \na \ffi (x)= 0  \}$ and observe that if $s_0$ is a singular value of $\ffi$ and thus the level set $\{ \ffi =  s_0 \}$ is nonregular, it happens that the nonempty closed set ${\rm Crit}(\ffi) \cap \{ \ffi =  s_0 \}$ has vanishing $(n-1)$-dimensional Hausdorff measure. In fact, since $\ffi$ is harmonic, one has by~\cite[Theorem 1.17]{Che_Nab_Val} that the Minkovsky dimension -- and thus also the Hausdorff dimension -- of its critical set ${\rm Crit}(\ffi) = \{ \na \ffi = 0\}$ is bounded above by $(n-2)$. It is worth noticing that the same conclusion about the Hausdorff dimension of ${\rm Crit}(\ffi)$ had been previously obtained in~\cite{Nadirashvili}. In particular, the previous formul\ae\, for the second fundamental form and mean curvature also make sense $\mathscr{H}^{n-1}$-almost everywhere in $\{ \ffi = s_0 \}$, namely on the relatively open set $\{ \ffi =  s_0\} \setminus {\rm Crit}(\ffi)$.

We conclude this section with some considerations about the geometry of $\pa M$ in the case of null Dirichlet boundary conditions for $\ffi$. We recall that, as $\ffi$ is nonconstant in $M$ and $\pa M$ is assumed to be smooth, the Hopf Lemma implies that $|\na \ffi|_g>0$ on $\pa M$. To continue, we observe that from the first equation in~\eqref{eq:pb_schw_reform} and from its traced version~\eqref{eq:tilde_R} it is immediate to deduce that
\begin{align}
\label{eq:hessg}
\nonumber |\nana \ffi |_\g^2 & \,\, = \,\, 
\tanh^2(\ffi) \, \bigg( \, \frac{2}{n-2} \, \Ricg (\na \ffi, \na\ffi) \, + \, |\Ricg|_g^2 \, - \, \frac{\Rg^2}{n-1} \,\bigg) \, .
\end{align}
In particular, we have that $\nana \ffi \equiv 0$ on $\{\ffi = 0 \}$. Hence, the boundary of $(M,g)$ is totally geodesic and $|\na \ffi|_g$ is constant on each connected component of $\pa M$. Notice once again that if $(M,g)$ is a cylinder and $\na \ffi$ is the splitting direction, one has 
\begin{equation*}
\Ricg (\na \ffi, \na\ffi) \, = \, 0 \qquad \hbox{and} \qquad |\Ricg|_g^2 \, = \, \frac{\Rg^2}{n-1}  \, ,
\end{equation*}
so that the right hand side of the above identity vanishes everywhere and $\ffi$ is an affine function in $(M,g)$, as expected.

\subsection{A conformal version of the Monotonicity-Rigidity Theorem.} 

\label{sub:reform}

We conclude this section by introducing the conformal analog of the functions
\begin{equation*}
%\label{eq:Up}
% , \, \qquad \hbox{with} \qquad 
t \,\, \longmapsto \,\, U_p(t) \, = \, \Big(\frac{2m}{1-t^2}\Big)^{\!\!\!\frac{(p-1)(n-1)}{(n-2)}}\!\!\!\!\!\! \int\limits_{ \{ u = t \}} \!\!\!\!  |\D u|^p \, \rmd \sigma ,
\end{equation*}
introduced in~\eqref{eq:Up}. To this aim,  we let $(M, \g, \ffi)$ be an aymptotically cylindrical solution to problem~\eqref{eq:pb_schw_reform} and we define, for $p \geq 0$, the functions $\Phi_p : [\ffi_0, +\infty) \longrightarrow \R$ as 
\begin{equation}
\label{eq:fip}
s \,\, \longmapsto \,\, \Phi_p(s) \,\,  =\!\!\!
\int\limits_{\{\ffi = s\}}\!\!\!
|\na \ffi|_g^p \,\,\rmd \sigma_{\!g} \, .
\end{equation}
As for the $U_p$'s, we observe that the $\Phi_p$'s are well defined. This is because $| \na \ffi|_g$ is bounded (see Lemma~\ref{le:bound}) and, by the results in~\cite{Hardt_Simon,Lin}, the hypersurface area of the level set is finite, due to the harmonicity and properness of $\ffi$. Before proceeding, it is worth noticing that, when $p=0$, the function 
$$
\Phi_0(s) \, = \!\!\!\int\limits_{\{\ffi = s\}}\!\!\!\!
\rmd \sigma_{\!g} \, = \, |\{ \ffi = s\}|_g ,
$$ 
coincides with the hypersurface area functional $|\{ \ffi = s\}|_g$ for the level sets of $\ffi$ inside the ambient manifold $(M,g)$. For $p=1$, it follows from $\Delta_g \ffi = 0$ and the Divergence Theorem that the function 
$$
\Phi_1(s) \, = \!\!\!\int\limits_{\{\ffi = s\}}\!\!\!\!|\na \ffi|_g \,\, 
\rmd \sigma_{\!g} \phantom{0000}
$$ 
is constant. We also observe that the asymptotic cylindrical behavior of $g$ and $\ffi$ implies that 
\begin{equation}
\label{eq:fip_lim}
\lim_{s\to +\infty}\Phi_p(s) \, = \, (2m)^{\!\frac{n-1-p}{n-2}} \, (n-2)^p \,  |\Sph^{n-1}| \,,
\end{equation}
where $m$ is the mass coefficient appearing in expansion~\eqref{eq:uexp}. For future convenience, we observe that the functions $U_p$ and $\Phi_p$ and their derivatives (when defined) are related as follows
\begin{eqnarray}
\label{eq:upfip1}
U_p(t) & = & \frac{(2m)^{\frac{(p-1)(n-1)}{(n-2)}}}{2^p} \, \,\Phi_p \left(\log\left[({1+t})/({1-t}) \right] \right) \, , \\
\label{eq:upfip2}
U'_p(t) & = & \frac{(2m)^{\frac{(p-1)(n-1)}{(n-2)}}}{2^{p-1} (1-t^2)} \, \,\Phi'_p \left(\log\left[({1+t})/({1-t})\right] \right)  \, , \\
\label{eq:upfip3}
U''_p(t) & = & \frac{(2m)^{\frac{(p-1)(n-1)}{(n-2)}}}{2^{p-2} (1-t^2)^2} \, \left(\, t \, \Phi'_p \left(\log\left[({1+t})/({1-t})\right] \right) +   \Phi''_p \left(\log\left[({1+t})/({1-t})\right] \right)  \, \right)\, .
%\bigg(\!\log\bigg(\frac{1+t}{1-t}\bigg)\bigg)
\end{eqnarray}
Using the above relationships the Monotonicity-Rigidity Theorem~\ref{thm:main} can be rephrased in terms of the functions $s \mapsto \Phi_p(s)$ as follows.

\begin{theorem}[Monotonicity-Rigidity Theorem -- Conformal Version]
\label{thm:main_conf}
Let $(M,g,\ffi)$ be an asymptotically cylindrical solution to problem~\eqref{eq:pb_schw_reform} with $0 \leq \ffi_0<+\infty$. For every $p \geq 1$ we let $\Phi_p : [\ffi_0, 1) \longrightarrow \R$ be the function defined  in~\eqref{eq:fip}. Then, the following properties hold true.
\begin{itemize}
\item[(i)] For every $p\geq 1$, the function $\Phi_p$ is continuous.  

\smallskip

\item[(ii)] For every $p \geq 3$, the function $\Phi_p$ is differentiable and the derivative satisfies, for every $s \in [\ffi_0,+ \infty)$,
\begin{align}
\label{eq:der_fip}
 \Phi_p'(s) \,\, 
& =  \,\, - \, (p-1)\!\!\!\! \int\limits_{\{\ffi = s \}}\!\!\!\!   {\, |\na\ffi|_g^{p-1}\, \Hg} \,\rmd\sigma_{\!g} \, \leq \, 0 \,.
\end{align}
where $\HHH_g$ is the mean curvature of the level set $\{\ffi=s\}$. Moreover, if $\Phi_p'(s)=0$, for some $s\in [\ffi_0, + \infty) $ and some $p \geq 3$, then $(M,g,\ffi)$ is isometric to one half round cylinder with totally geodesic boundary.

\smallskip

\item[(iii)] Suppose that $\ffi=0$ at $\pa M$. Then $\Phi_p'(0) = \lim_{s \to 0^+} \Phi_p'(s)= 0$, for every $p\geq 3$. In particular, setting $\Phi''_p(0) = \lim_{s \to 0^+}\Phi_p'(s)/s$, we have that for every $p \geq 3$, it holds
\begin{equation}
\label{eq:der2_fip}
\phantom{\qquad} \Phi_p''(0) \, = \,   ({p-1})  
\int\limits_{\pa M}
{|\na \ffi|_\g^{p-2}\,\,\Ric_g(\nu_g,\nu_g) }\,\,\rmd\sigma_g \,\, \leq \,\, 0 \, ,
\end{equation}
where $\nu_g = \na\ffi/|\na \ffi|_g$ is the inward pointing unit normal of the boundary $\pa M$.
Moreover, if $\Phi_{p}''(0) =0$ for some $p \geq 3$, then $(M,g,\ffi)$ is isometric one half round cylinder with totally geodesic boundary.
\end{itemize}
\end{theorem}
\begin{remark}
\label{rem:der2}
To see the equivalence of Theorem~\ref{thm:main}-(iii) and Theorem~\ref{thm:main_conf}-(iii) one has to observe that $U_p''(0)$ and $\Phi_p''(0)$ are proportional, by identity~\eqref{eq:upfip3}. Moreover, from the computations in subsection~\ref{sub:further}, it follows that
\begin{equation*}
U_p''(0) \, = \,  - \Big(\frac{p-1}{2} \Big) \, (2m)^{\!\frac{(p-1)(n-1)}{(n-2)}} \!\!\!\int\limits_{\pa M}
\!
  |\DDD u|^{p-2}
\left[\, \RRR^{\pa M}\! - 
4 \, \Big(\frac{n-1}{n-2}\Big)
\, { \,|\DDD u|^2}
\, \right]  \rmd\sigma .
\end{equation*}
Combining these facts, it is easy to realize that~\eqref{eq:der2up} and~\eqref{eq:der2_fip}, as well as the respective rigidity statements, are equivalent.
\end{remark}

Since it is now clear that Theorem~\ref{thm:main} is completely equivalent to Theorem~\ref{thm:main_conf}, the rest of the paper is devoted to the proof of Theorem~\ref{thm:main_conf}. We will also prove at the same time the following theorem, which is the conformal version of Theorem~\ref{thm:refined}.

\begin{theorem}
\label{thm:refined_conf}
Let $(M,g,\ffi)$ be an asymptotically cylindrical solution to problem~\eqref{eq:pb_schw_reform} with $0 \leq \ffi_0<+\infty$. Let $\overline{\ffi}_0 \in [\ffi_0,+ \infty)$ be such that $|\na \ffi|_g>0$ in the region $\{  \overline{\ffi}_0 < \ffi < + \infty\}$ and let $\Phi_p~\!:~\![\ffi_0, +\infty) \longrightarrow \R$ be the function defined in~\eqref{eq:fip}. Then, the following properties hold true.
\begin{itemize}
\item[(i)] For $p\geq 0$, the function $\Phi_p$ is continuous and differentiable in $(\overline{\ffi}_0,+\infty)$.  

\smallskip

\item[(ii)] For $p \geq 2-1/(n-1)$, we have that
$\Phi_p'(s) \leq 0 $ for every $s \in (\overline{\ffi}_0,+\infty)$. Moreover, if there exists $s \in (\overline{\ffi}_0, +\infty) $ such that $\Phi_{p}'(s) = 0$ for some $p \geq 2-1/(n-1)$, then $(M,g,\ffi)$ is isometric to one half round cylinder. 
\end{itemize}
\end{theorem}

\section{Integral identities}
\label{sec:integral}
%%%%%%%%%%%%%%%%%%%%%%%%%%%%%%%%%%%%%%%%%%%%%%%
%%%%%%%%%%%%%%%%%%%%%%%%%%%%%%%%%%%%%%%%%%%%%%%

In this section, we derive some integral identities that will be used to analyze the properties of the functions $s \mapsto \Phi_p(s)$ introduced in~\eqref{eq:fip}.
To obtain the first identity, we are going to exploit
the equation $\Delta_g \, \ffi = 0$ in combination with Lemma~\ref{le:bound}.
\begin{proposition}
\label{prop:byparts}
Let $(M, \g, \ffi)$ be an aymptotically cylindrical solution to problem~\eqref{eq:pb_schw_reform},
then, for every $p\geq 1$ and for every $ s \in [\ffi_0, + \infty) \cap (0,+\infty)$, we have
\begin{equation}
\label{eq:id_byparts_bis}
\int\limits_{\{\ffi=s\}} \!\!
\frac{   |\na \ffi|_\g^{p} }{ \sinh(s)  }
\,\,\rmd\sigma_{\!g}  
\,\, =  \!\!
\int\limits_{\{\ffi> s\}} \!\!\frac{  |\na \ffi|_\g^{p-3}  
\Big(  \coth(\ffi) \, |\na\ffi|_g^{4}   \, 
- \,  (p-1)\,  \nana \ffi (\na\ffi, \!\na\ffi)  \Big) }{\sinh(\ffi)}
\,\,\rmd\mu_g \, .
\end{equation}
\end{proposition}
\begin{remark}
\label{rem:a_e1}
Before proceeding with the proof of the proposition, it is worth pointing out that the left hand side of~\eqref{eq:id_byparts_bis} is well defined also when $s$ is not a regular value of $\ffi$. In fact, since the function $\ffi$ is harmonic and proper, by the already cited results in~\cite{Hardt_Simon,Lin,Nadirashvili} and~\cite{Che_Nab_Val} one has that the $(n-1)$-dimensional Hausdorff measure of the level sets of $\ffi$ is finite. Moreover, the Hausdorff dimension of the critical set ${\rm Crit}(\ffi)$ of the function $\ffi$ is bounded above by $n-2$. In particular, the density that relates the volume element $\rmd\sigma_{\!g}$ -- which is well defined only on the regular portion of the level set -- to the everywhere defined volume element $\rmd \mathscr{H}^{n-1}$ is well defined and bounded $\mathscr{H}^{n-1}$-almost everywhere on $\{\ffi = s\}$. Hence, the integral make sense.
\end{remark}

\begin{proof}
For the sake of simplicity, we drop the subscript $\g$ in the notation of this proof. To prove identity~\eqref{eq:id_byparts_bis} when $s>0$ is a regular value of $\ffi$, we start from the formula
\begin{equation}
\label{eq:diver}
{\rm div}  \bigg( \frac{|\na \ffi|^{p-1} \, \na \ffi   }{\sinh(\ffi)} \bigg) \, = \, \frac{  |\na \ffi|^{p-3}  
\Big(     (p-1)\,  \nana \ffi (\na\ffi, \!\na\ffi) \, 
- \, \coth(\ffi) \, |\na\ffi|^{4}  \Big) }{\sinh(\ffi)} \, ,
\end{equation}
which follows from a direct computation, using the fact that $\ffi$ is harmonic. Next we integrate the above formula by parts using the Divergence Theorem in $\{s < \ffi < S \}$, where $S$ is so large that the level set $\{\ffi = S \}$ is regular. This gives
\begin{align}
\label{eq:byparts_fin}
 \int\limits_{\{s <\ffi < S\}} \!\!\!\!\!\!
\frac{  |\na \ffi|^{p-3}  
\Big(     (p-1)\,  \nana \ffi (\na\ffi, \!\na\ffi) \, 
- \, \coth(\ffi) \, |\na\ffi|^{4}  \Big) }{\sinh(\ffi)}
\,\,\rmd\mu \,\,= & \\
\nonumber \,=\,\,
\!\!\!\!
\int\limits_{\{\ffi=S\}}
\!\!\!
\frac{   |\na\ffi|^{p-1} \big\langle{\na\ffi}
\,\big|\,{\rm n}\big\rangle}{  \sinh(\ffi)  }    &   \,\,\rmd\sigma 
\,\,+
\!\!\int\limits_{\{\ffi=s\}}
\!\!\!
\frac{    |\na\ffi|^{p-1} \big \langle{\na \ffi}
\,\big|\, {\rm n}\big\rangle}{  \sinh(\ffi)  }       \,\,\rmd\sigma \, ,
\end{align}
where ${\rm n}$ is the outer unit normal. In particular, one has that ${\rm n} = -\na\ffi/|\na\ffi|$ on $\{\ffi=s\}$ and ${\rm n} =\na\ffi/|\na\ffi|$
on $\{\ffi=S\}$.
On the other hand, thanks to Lemma~\ref{le:bound}, it is immediate to deduce that 
$$
\lim_{S \to +\infty} \!\!\int\limits_{\{\ffi=S\}} \!\!\! \frac{|\na \ffi|_g^p}{\sinh(\ffi)} \,\,\rmd \sigma_{\!g} \, = \, 0 \,. 
$$ 
The statement of the proposition then follows at once. 

In the case where $s>0$ is a singular value of $\ffi$, we need to apply a slightly refined version of the Divergence Theorem, namely Theorem~\ref{thm:div} in the Appendix, in order to perform the integration by parts which leads to identity~\eqref{eq:byparts_fin}. The rest of the proof is then identical to what we have done for the regular case. According to the notations of Theorem~\ref{thm:div}, we set 
\begin{eqnarray*}
X  =  \frac{|\na\ffi|^{p-1}\na \ffi}{\sinh(\ffi)} \qquad \hbox{and} \qquad E = \{ s < \ffi < S\}\, .
\end{eqnarray*}
so that $\pa E = \{\ffi = s \} \sqcup \{ \ffi = S \}$. As we have already observed, since $\ffi$ is harmonic and proper, the $(n-1)$-dimensional Hausdorff measure of $\pa E$ is finite. As usual, we denote by ${\rm Crit}(\ffi) = \{  x\in M \, | \, \na \ffi (x)= 0  \}$ the set of the critical points of $\ffi$, and we set $\Sigma = \pa E \cap {\rm Crit}(\ffi)$ and $\Gamma = \pa E \setminus \Sigma$, so that $\pa E$ can be written as the disjoint union of $\Sigma$ and $\Gamma$. Moreover, up to choosing $S$ large enough, we can suppose that $\{\ffi = S \} \cap {\rm Crit}(\ffi)$ is empty. 
Since $|\na \ffi| >0$ in $\Gamma$, it is easy to check that, for every $x \in \Gamma$, there exists an open neighborhood $U_x$ of $x$ in $M$ such that $\pa E \cap U_x = \Gamma \cap U_x$ is a smooth regular hypersurface. 

To apply Theorem~\ref{thm:div}, we also need to check that $\mathscr{H}^{n}(B_\ep(\Sigma)) = o(\ep)$, as $\ep \to 0$, where, for an arbitrary set $A$, the $\ep$-neighborhood $B_\ep(A)$ of $A$ is given by $B_\ep (A)= \bigcup_{x \in A} B_\ep(x)$. To see this, we recall from~
%\cite[CHEEGER, NABER, VALTORTA]{} that the Minkovsky dimension of $\mathscr{C}(\ffi)$ is at most $n-2$. More precisely, from~
\cite[Theorem 1.17]{Che_Nab_Val} that, since $\ffi$ is harmonic, for every $\eta >0$ there exists a constant $C_\eta>0$ such that 
\begin{equation*}
\mathscr{H}^{n} \big( B_\ep\big(\pa E \cap {\rm Crit}(\ffi) \big) \big) \, \leq \, C_\eta \, \ep^{2-\eta} \, , \phantom{0000000000000}
\end{equation*}
for every $\ep>0$. In particular, we have that
\begin{equation*}
\phantom{0000000000}\mathscr{H}^{n} \big( B_\ep\big( \Sigma \big) \big) \, = \, \mathscr{H}^{n} \big( B_\ep\big( \pa E \cap {\rm Crit}(\ffi) \big) \big) \, 
 = \,  o(\ep) \, , \qquad \hbox{as $\ep \to 0$.}
\end{equation*}

To apply Theorem~\ref{thm:div} we finally need to check that 
for $p\geq 1$ the vector field $X$ is bounded with bounded divergence in $\overline{E}$. But this can be easily deduced from formula~\eqref{eq:diver} and Lemma~\ref{le:bound} and the proof is complete.
\end{proof}

To obtain the second relevant integral identity, we are going to combine the equations of system~\eqref{eq:pb_schw_reform} with the Bochner formula
\begin{equation*}
\frac12\Deg|\na\ffi|_\g^2 \, = \, \big|\nana\ffi\big|^2_\g
+\Ricg(\na\ffi,\na\ffi)
+\big\langle\na\Deg\ffi \, \big| \,\na\ffi\big\rangle_{\g} \,.
\end{equation*} 
Using the fact that $\ffi$ is $\g$-harmonic together with the first equation in~\eqref{eq:pb_schw_reform}, it is easily seen that the Bochner formula reduces to
%\eqref{eq:pb_schw_reform},
%then the Bochner formula \eqref{eq:Boch_pura} reduces to
\begin{equation}
\label{eq:Bochner_schw}
\Delta_\g|\na \ffi|_\g^2 \, - \, \big\langle\na|\na \ffi|^2_\g \, \big| \,\na \log \big( \sinh( \ffi ) \big) 
\big\rangle_{\!\!\g} \,
= \,
2 \, |\na^2 \ffi|^2_\g \, .
\end{equation} 
For every $p \geq 3$, we compute
%Whenever $|\na \ffi|_\g>0$, straightforward computations give, for every $p \in\R$, 
\begin{align*}
\na |\na \ffi|_\g^{p-1} & \,= \, \Big(\frac{p-1}{2}\Big)  |\na \ffi|_\g^{p-3} \, \na |\na \ffi|_\g^2 \, ,\\
\Deg |\na \ffi|_\g^{p-1} &\, = \, \Big(\frac{p-1}{2}\Big)  |\na \ffi|_\g^{p-3} \, {\Deg |\na \ffi|^2} +  (p-1)(p-3) \, |\na \ffi|_\g^{p-3} \, \big| \na |\na \ffi |_\g   \big|_\g^2 \, .
\end{align*}
We notice {en passant} that whenever $|\na \ffi|_g>0$ the above formul\ae\ make sense for every $p \geq 0$.
These  identities, combined with~\eqref{eq:Bochner_schw}, lead to
\begin{equation}
\label{eq:Bochner_schw_p}
\Delta_\g|\na \ffi|_\g^{p-1}\, - \, \big\langle\na|\na \ffi|^{p-1}_\g \, \big| \,\na \log \big( \sinh( \ffi ) \big) 
\big\rangle_{\!\!\g} =  (p-1)   |\na \ffi|_\g^{p-3}  
\Big(   \,  \big|\nana \ffi\big|_\g^2
 +  \, (p-3) \, \big| \na |\na \ffi |_\g   \big|_\g^2   \, \Big)  \,.
\end{equation}
Obviously, for $p=3$, the above formula coincides with~\eqref{eq:Bochner_schw}. We notice that the differential operator appearing on the left hand side of the above expression, namely
\begin{equation*}
\Delta_g \cdot \,\, - \, \,\big\langle\na \cdot \, \big| \,\na \log \big( \sinh( \ffi ) \big) 
\big\rangle_{\!\!\g} \, ,
\end{equation*}
is a drifted Laplace-Beltrami operator and thus it is formally self-adjoint with respect to the weighted measure $\big( 1/\sinh(\ffi) \big) \rmd\mu_g$.
%\begin{equation*}
%\frac{\rmd\mu_g}{\sinh(\ffi)} \, .
%\end{equation*}
Integrating by parts the identity~\eqref{eq:Bochner_schw_p} with respect to this weighted measure, we obtain the following proposition, which is the main result of this section.

\begin{proposition}
\label{prop:cyl}
Let $(M, \g, \ffi)$ be an aymptotically cylindrical solution to problem~\eqref{eq:pb_schw_reform} with $0 \leq \ffi_0<+\infty$.
Then, for every $p\geq 3$ and every $ s \in [\ffi_0, + \infty)$, we have
\begin{equation}
\label{eq:id_byparts}
\int\limits_{\{\ffi=s\}} \!\!
\frac{   |\na \ffi|_\g^{p-1}\,\HHH_g }{ \sinh(s)  }
\,\,\rmd\sigma_{\!g}  
\,\, =  \!\!
\int\limits_{\{\ffi> s\}} \!\!
\frac{  |\na \ffi|_\g^{p-3}  
\Big(   \,  \big|\nana \ffi\big|_\g^2
 +  \, (p-3) \, \big| \na |\na \ffi |_\g   \big|_\g^2   \, \Big)       }{\sinh(\ffi)}
\,\,\rmd\mu_g \, .
\end{equation}
Moreover, if there exists $s_0\in [\ffi_0, + \infty) \cap (0,+\infty)$ such that 
%$\{ \ffi = s_0\} \subset M$ and 
\begin{equation}
\label{eq:condition_cyl}
\int\limits_{\{\ffi=s_0\}}
{   |\na \ffi|_\g^{p_0-1}\,\HHH_g }
\,\rmd\sigma_{\!g}  \, \leq \, 0 \, ,
\end{equation}
for some $p_0 \geq 3$, then the manifold $(\{ \ffi \geq s_0\} , g)$ is isometric to one half round cylinder and $\ffi$ is an affine function.
\end{proposition}

\begin{remark}
\label{rem:a_e}
Translating Remark~\ref{rem:uno} in terms of the conformally related quantities, it is easy to realize that the integral on the left hand side of~\eqref{eq:id_byparts} is well defined also when $s$ is a singular value of $\ffi$.
\end{remark}
\begin{remark}
\label{rem:analytic}
We observe that since the {\em static solution} $(M, g_0,u)$ to problem~\eqref{eq:pb_static_vacuum} is analytic (see for example~\cite{Chr}), the solution $(M,g,\ffi)$ to problem~\eqref{eq:pb_schw_reform} coming from $(M,g_0,u)$ through~\eqref{eq:g_schw} and~\eqref{eq:newvariable} is analytic as well. Hence, the conclusion of the rigidity statement in Proposition~\ref{prop:cyl} can be made stronger in the sense that if $(\{ \ffi \geq s_0\} , g)$ is isometric to one half round cylinder, then the entire manifold $(M,g)$ must be isometric to one half round cylinder and the corresponding {\em static solution} $(M, g_0,u)$ must be rotationally symmetric and thus isometric to the {\em Schwarzschild solution}~\eqref{eq:sol_schwarz}.
\end{remark}

\begin{proof} 
For the sake of simplicity, we drop the subscript $\g$ in the notation of this proof. In the same spirit as in Proposition~\ref{prop:byparts}, we start by considering the case where 
the level set $\{ \ffi = s \}$ is regular, meaning that $|\na \ffi| >0$ on $\{ \ffi = s\}$. We observe that, whenever $\ffi > 0$, one can write
\begin{align}
\label{eq:id_div}
{\rm div} \bigg( \frac{\na|\na\ffi|^{p-1}}
{\sinh(\ffi)} \bigg)
&\, = \,  \frac{\Delta|\na \ffi|^{p-1}\, - \, \big\langle\na|\na \ffi|^{p-1}\, \big| \,\na \log \big( \sinh( \ffi ) \big) 
\big\rangle}{\sinh(\ffi)} \nonumber \\
&\, = \,  (p-1) \,  \frac{ |\na \ffi|^{p-3}  
\Big(   \,  \big|\nana \ffi\big|^2
 +  \, (p-3) \, \big| \na |\na \ffi |   \big|^2   \, \Big) }{\sinh(\ffi)} \, ,
\end{align}
where in the second equality we used equation~\eqref{eq:Bochner_schw_p}. Since for every large enough $S>0$ the level set $\{ \ffi = S \}$ is regular 
(see the discussion after formula~\eqref{ass:AC})
we integrate by parts the above identity, obtaining 
\begin{multline*}
(p-1) \!\!\!\!\!\! \int\limits_{\{s <\ffi < S\}} \!\!\!\!\!\!
\frac{  |\na \ffi|^{p-3}  
\Big(   \,  \big|\nana \ffi\big|^2
 +  \, (p-3) \, \big| \na |\na \ffi |   \big|^2   \, \Big)       }{\sinh(\ffi)}
\,\,\rmd\mu \,\,=  \\
\,=\,\,
\!\!\!\!
\int\limits_{\{\ffi=S\}}
\!\!\!
\frac{    \big\langle{\na|\na\ffi|^{p-1}}
\,\big|\,{\rm n}\big\rangle}{  \sinh(\ffi)  }       \,\,\rmd\sigma
\,\,+
\!\!\int\limits_{\{\ffi=s\}}
\!\!\!
\frac{    \big\langle{\na|\na\ffi|^{p-1}}
\,\big|\, {\rm n}\big\rangle}{  \sinh(\ffi)  }       \,\,\rmd\sigma \, ,
\end{multline*}
where ${\rm n}$ is the outer $\g$-unit normal
of the set $\{s \leq \ffi \leq S \}$ at its boundary. In particular, one has that ${\rm n} = -\na\ffi/|\na\ffi|$ on $\{\ffi=s\}$ and ${\rm n} =\na\ffi/|\na\ffi|$
on $\{\ffi=S\}$. On the other hand, from the second formula in \eqref{eq:formula_curvature} it is easy to deduce that
$$
\langle\na|\na\ffi|^{p-1}|\na\ffi\rangle
\, = \, (p-1) \, |\na \ffi|^{p-3} \, \nana\ffi(\na\ffi,\na\ffi) \, = \, - (p-1) |\na \ffi|^p \, \HHH \, . 
$$
Therefore, we have obtained
\begin{equation}
\label{eq:int_part_fin}
\int\limits_{\{ s<\ffi< S\}}
\!\!\!\!\!\!\frac{  |\na \ffi|^{p-3}  
\Big(   \,  \big|\nana \ffi\big|^2
 +  \, (p-3) \, \big| \na |\na \ffi |   \big|^2   \, \Big)       }{\sinh(\ffi)}
\,\,\rmd\mu \,\, = \!\! 
\int\limits_{\{\ffi=s\}} \!\!\!\!
\frac{   |\na \ffi|^{p-1}\,\HHH }{ \sinh{(s)}  }
\,\,\rmd\sigma  
\,  - \!\!\! \int\limits_{\{\ffi=S\}}\!\!\!\!
\frac{   |\na \ffi|^{p-1}\,\HHH }{ \sinh{(S)}  }
\,\,\rmd\sigma \, .
\end{equation}
In order to obtain identity~\eqref{eq:id_byparts} it is sufficient to show that the last term on the right hand side tends to zero as $S\to +\infty$. To see this, we observe that
$|\na \ffi|^{p-1} \Ho \leq |\na \ffi|^{p-2} |\nana \ffi|$ and thus it is uniformly bounded, by estimate~\eqref{eq:cylbounds} in Lemma~\ref{le:bound}. The same estimate provides us with a uniform bound for the area of the level sets of $\ffi$. Hence, it is easy to arrive to the desired conclusion. This completes the proof of the proposition in the case where
$\{ \ffi =s\}$ is regular. 

In the case where $s>0$ is a singular value of $\ffi$, we need to apply a slightly refined version of the Divergence Theorem, namely Theorem~\ref{thm:div} in the Appendix, in order to perform the integration by parts which leads to identity~\eqref{eq:int_part_fin}. The rest of the proof is identical to what we have done for the regular case. According to the notations of Theorem~\ref{thm:div}, we set 
\begin{eqnarray*}
X  =  \frac{\na|\na\ffi|^{p-1}}{\sinh(\ffi)} \qquad \hbox{and} \qquad E= \{ s < \ffi < S\}\, .
\end{eqnarray*}

As it is easy to realize, the same considerations as in the proof of Proposition~\ref{prop:byparts} apply to the present situation. The only difference amounts to check that for $p\geq 3$ the vector field $X$ is bounded with bounded divergence in $\overline{E}$. To see this, we observe that the definition of $X$ combined with Kato inequality implies
$$
|X| \, \leq \, (p-1) \, \frac{|\na \ffi|^{p-2} \, |\nana \ffi|}{\sinh(\ffi)} \, . \phantom{0000\,}
$$
Moreover, using equation~\eqref{eq:id_div} together with Kato inequality, it is easy to deduce that
$$
|{\rm div} X| \, \leq \, (p-1) (p-2) \, \frac{|\na \ffi|^{p-3} \,  |\nana \ffi|^2}{\sinh(\ffi)} \, .
$$
Since $p \geq 3$, the claim follows now directly from Lemma~\ref{le:bound}. Hence, all the hypotheses of Theorem~\ref{thm:div} are in force and we can integrate ${\rm div} X$ by parts, obtaining
\begin{equation*}
\int\limits_{\{s < \ffi < S\}} \!\!\!\!\!\!
{\rm div} \bigg( \frac{\na|\na\ffi|^{p-1}}
{\sinh(\ffi)} \bigg)
\,\rmd\mu  \,\,\, = \,
\!\!\!\!
\int\limits_{\{\ffi=S\}}
\!\!\!
\frac{    \big\langle{\na|\na\ffi|^{p-1}}
\,\big|\,{\rm n}\big\rangle}{  \sinh(\ffi)  }       \,\,\rmd\sigma
\,\,\,+
\!\!\!\!\!\!\!\!\!\int\limits_{\{\ffi=s\} \setminus {\rm Crit}(\ffi)}
\!\!\!\!\!\!\!\!\!\!\!
\frac{    \big\langle{\na|\na\ffi|^{p-1}}
\,\big|\, {\rm n}\big\rangle}{  \sinh(\ffi)  }       \,\,\rmd\sigma \, .
\end{equation*}
Taking into account Remark~\ref{rem:a_e} and expression~\eqref{eq:id_div}, we have that identity~\eqref{eq:int_part_fin} holds true also in the case where $s$ is a singular value of $\ffi$. 

To prove the second part of the statement, we observe that from~\eqref{eq:id_byparts} and~\eqref{eq:condition_cyl} one immediatetely gets $\nana \ffi \equiv 0$ in $\{ \ffi \geq s_0 \}$. Moreover, by the asymptotic behavior of $\ffi$, one has that $|\na \ffi|$ is a positive constant in that region. Hence $\na \ffi$ is a nontrivial parallel vector field and by~\cite[Theorem 4.1-(i)]{Ago_Maz_1} we deduce that the Riemannian manifold $(\{ \ffi \geq s_0 \}, \g)$ is isometric to the manifold $\{ \ffi = s_0 \} \times [s_0, + \infty)$ endowed with the product metric $d\varrho \otimes d\varrho + \g_{|\{ \ffi = s_0 \}}$. Here $\varrho$ represents the distance to $\{ \ffi = s_0\}$ and the function $\ffi$ itself can be expressed as an affine function of $\varrho$ in $\{ \ffi \geq s_0 \}$, that is $\ffi = s_0 + \varrho \, |\na\ffi|_\g $, where $|\na \ffi|_\g$ is a positive constant. Combining the product structure with the asymptotic behavior~\eqref{eq:ACg}, as already observed in Remark~\ref{rem:pr_round}, we arrive at $\g_{|\{ \ffi = s_0 \}} = (2m)^{2/(n-2)} g_{\Sph^{n-1}}$, which is the desired conclusions.
\end{proof}

As an immediate consequence of the above proposition, we obtain the following corollary.

\begin{corollary}
\label{cor:riccinunu}
Let $(M, \g, \ffi)$ be an aymptotically cylindrical solution to problem~\eqref{eq:pb_schw_reform} with $\ffi_0 = 0$.
Then, for every $p\geq 3$, we have
\begin{equation}
\label{eq:id_byparts_ric}
-\!\!\int\limits_{\pa M} \!
|\na \ffi|_g^{p-2} \, \Ric_g (\nu_g, \nu_g)
\,\,\rmd\sigma_{\!g}  
\,\, =  \,
\int\limits_{M} \,\,
\frac{  |\na \ffi|_\g^{p-3}  
\Big(   \,  \big|\nana \ffi\big|_\g^2
 +  \, (p-3) \, \big| \na |\na \ffi |_\g   \big|_\g^2   \, \Big)       }{\sinh(\ffi)}
\,\,\rmd\mu_g \, .
\end{equation}
Moreover, if the left hand side of the above formula vanishes for some $p_0 \geq 3$, then the manifold $(M, g)$ is isometric to one half round cylinder and $\ffi$ is an affine function.
\end{corollary}
\begin{proof}
It is sufficient to observe that since $
|\na \ffi|_g^{p-1} \Hg \, = \, - \, \tanh(\ffi) \, |\na \ffi|_g^{p-2} \,\Ric_g (\nu_g, \nu_g)$, one has that
\begin{equation*}
\lim_{s \to 0^+} \!\! \int\limits_{\{\ffi=s\}} \!\!
\frac{   |\na \ffi|_\g^{p-1}\,\HHH_g }{ \sinh(s)  }
\,\,\rmd\sigma_{\!g} \,\, = \,\, -\!\!\int\limits_{\pa M} \!
|\na \ffi|_g^{p-2} \, \Ric_g (\nu_g, \nu_g)
\,\,\rmd\sigma_{\!g} \,.
\end{equation*}
The result is then a straightforward consequence of Proposition~\ref{prop:cyl}.
\end{proof}

We observe that a slightly stronger version of rigidity result contained in the previous proposition holds if $|\na \ffi|_g >0$ in the region $\{  \ffi \geq s_0\}$,  $\{ \ffi = s_0 \}$ being the level set where inequality~\eqref{eq:condition_cyl} is in force. In fact, in this case one can allow the exponent $p$ to vary in a larger range. It is worth pointing out that the condition $|\na \ffi| >0$ on the level sets which are sufficiently close to the end of $M$.

\begin{proposition}
\label{prop:cyl_bis}
Let $(M,g,\ffi)$ be an asymptotically cylindrical solution to problem~\eqref{eq:pb_schw_reform} with $0 \leq \ffi_0<+\infty$. Let $\overline{\ffi}_0 \in [\ffi_0,+\infty)$ be such that $|\na \ffi|_g>0$ in the region $\{  \overline{\ffi}_0 < \ffi <+ \infty\}$.
Then, for every $p\geq0$ and for every $s \in [\overline{\ffi}_0 , + \infty)$, formul\ae\ \eqref{eq:id_byparts_bis} and~\eqref{eq:id_byparts} hold true.
Suppose in addition that there exists $s_0\in [\overline{\ffi}_0, + \infty) \cap (0,+\infty)$ such that
\begin{equation}
\label{eq:condition_cyl_bis}
\int\limits_{\{\ffi=s_0\}}
{   |\na \ffi|_\g^{p_0-1}\,\HHH_g }
\,\rmd\sigma_g  \, \leq \, 0 \, ,
\end{equation}
holds for some $p_0 \geq 2-1/(n-1)$, then the manifold $(\{ \ffi \geq s_0\} , g)$ is isometric to one half round cylinder.
%\bigskip
%\bigskip
%Let $(M, \g, \ffi)$ be an asymptotically cylindrical solution to problem~\eqref{eq:pb_schw_reform} al let $s_0>0$ be a regular value of the function $\ffi$ such that 
%%$\{ \ffi = s_0\} \subset M$ and 
%$|\na \ffi|_\g >0$ in $\{ \ffi \geq s_0\}$. If in addition the inequality
%\begin{equation}
%\label{eq:condition_cyl_bis}
%\int\limits_{\{\ffi=s_0\}}
%{   |\na \ffi|_\g^{p_0-1}\,\HHH_g }
%\,\rmd\sigma_g  \, \leq \, 0 \, ,
%\end{equation}
%holds for some $p_0 \geq (2n-3)/(n-1)$, then the manifold $(\{ \ffi \geq s_0\} , g)$ is isometric to one half round cylinder.
\end{proposition}
\begin{proof} 
To prove the first part of the statement, namely the validity of formul\ae\ \eqref{eq:id_byparts_bis} and~\eqref{eq:id_byparts}, we observe that if $|\na \ffi|>0$ then the left hand side in equations~\eqref{eq:diver} and~\eqref{eq:id_div} are well defined for every $p \geq 0$ in $\{ \ffi \geq \overline{\ffi}_0\}$. In particular, one can perform the same integration by parts as in the previous propositions, obtaining identities~\eqref{eq:id_byparts_bis} and~\eqref{eq:id_byparts} for every $p \geq 0$. 

To prove the rigidity statement, we proceed in the same spirit as in the proof of Proposition~\ref{prop:cyl} and we observe
that if $p_0 \geq 2-1/(n-1)$, then $(p_0 - 3) \geq - n/(n-1)$ and one has by~\eqref{eq:id_byparts} and~\eqref{eq:condition_cyl_bis} that 
$$
|\nana \ffi|^2
 +  \, (p_0-3) \, | \na |\na \ffi |  |^2 \, \equiv \, 0 \, ,
$$
in $\{ \ffi \geq s_0\}$. In fact, the refined Kato inequality for harmonic functions gives
$$
|\nana \ffi|^2 \, \geq \, \frac{n}{n-1} \, | \na |\na \ffi ||^2 \, ,
$$
whenever $|\na \ffi| >0$, and thus the integrand on the right hand side of~\eqref{eq:id_byparts} is nonnegative. Now, if $(p_0 -3) > - n/(n-1)$, it is immediate to conclude that $|\nana \ffi| \equiv 0 \equiv |\na|\na\ffi||$ in $\{ \ffi \geq s_0\}$ and the thesis follows by the same arguments as in Proposition~\ref{prop:cyl}. In the limiting case where $(p_0-3) = -n/(n-1)$, one has that  
\begin{equation}
\label{eq:ref_kato}
|\nana \ffi|^2 \, = \, \frac{n}{n-1} \, | \na |\na \ffi ||^2 \, ,
\end{equation}
in $\{ \ffi \geq s_0\}$. Following the proof of~\cite[Proposition 5.1]{Bou_Car} it is possible to deduce that $|\na \ffi|^2$ is constant along the level sets of $\ffi$ and thus that the metric $g$ has a warped product structure in this region, namely 
\begin{equation}
\label{eq:warped}
g \,\, = \,\, d\varrho \otimes d\varrho + \eta^2(\varrho) \,\, \g_{|\{ \ffi = s_0 \}} \, ,
\end{equation}
for some positive warping function $\eta = \eta (\varrho)$. Moreover, $\ffi$ and $\rho$ satisfy the relationship
$$
\ffi (q) \, = \, s_0 \, + \, \kappa \!\!\int\limits_{0}^{\varrho(q)} \!\!\! \frac{{\rm d} \tau}{\eta^{n-1}(\tau)} \,, 
$$
for every $q \in \{ \ffi \geq s_0 \}$ and some $\kappa \geq 0$. In particular, $\ffi$ and $\varrho$ share the same level sets and, by formula~\eqref{eq:warped}, these are totally umbilic. In fact one has
\begin{equation*}
\chg_{ij} \, = \, \frac12 \frac{\pa \g_{ij}}{ \pa \varrho} \, = \, \frac{d \log \eta}{d \varrho} \, g_{ij} \, .
\end{equation*}
As a consequence, the mean curvature is constant along each level set of $\ffi$. Applying formula~\eqref{eq:id_byparts}, 
to every level set $\{\ffi = s \}$ with $s \geq s_0$ and $p_0= 2-1/(n-1)$, one gets  
$$
\HHH \! \! \int \limits_{\{ \ffi = s \}} \!\!\! |\na \ffi|^{\frac{n-2}{n-1}} \, \rmd \sigma  \, = \, 0 \, ,
$$
since the right hand side of~\eqref{eq:id_byparts} is always zero, due of~\eqref{eq:ref_kato}. This implies in turn that all the level sets $\{\ffi = s \}$ with $s\geq s_0$ are minimal and thus totally geodesics. From $\HHH \equiv 0$ one can also deduce that $\langle \na |\na\ffi|^2 \, \big| \, \na \ffi\rangle \equiv 0$ in $\{ \ffi \geq s_0\}$. Hence, $|\na \ffi|^2$ is constant in $\{ \ffi \geq s_0 \}$, and thus, by equation~\eqref{eq:Bochner_schw_p}, one gets $|\nana \ffi| \equiv 0$ in $\{\ffi \geq s_0 \}$. Again, the conclusion follows arguing as in Proposition~\ref{prop:cyl}.
\end{proof}

\section{Proof of Theorem~\ref{thm:main_conf} and Theorem~\ref{thm:refined_conf}}
\label{sec:proofs}

%%%%%%%%%%%%%%%%%%%%%%%%%%%%%%%%%%%%%%%%%%%%%%
%%%%%%%%%%%%%%%%%%%%%%%%%%%%%%%%%%%%%%%%%%%%%%

Building on the analysis of the previous section, we are now in the position to prove Theorem~\ref{thm:main_conf} and Theorem~\ref{thm:refined_conf}, which in turn imply Theorem~\ref{thm:main} and Theorem~\ref{thm:refined}, respectively.

\subsection{Continuity.}
\label{continuity}
We claim that under the hypotheses of Theorem~\ref{thm:main_conf} the function $\Phi_p$ is continuous, for $p\geq 1$. We first observe that since we are assuming that the boundary $\pa M$ is a regular level set of $\ffi$, the function $s \mapsto \Phi_p(s)$ can be described in term of an integral depending on the parameter $s$, provided $s \in [\ffi_0, \ffi_0 + 2\ep)$ with $\ep>0$ sufficiently small. In this case, the continuous dependence on the the parameter $s$ can be easily checked using standard results from classical differential calculus. Thus, we leave the details to the interested reader and we pass to consider the case where $s \in (\ffi_0 + \ep, + \infty)$. Thanks to Proposition~\ref{prop:byparts}
one can rewrite expression~\eqref{eq:fip} as 
\begin{align}
\label{eq:fip_div}
\Phi_p(s)\,\,
&= \,\, - \, \sinh(s) \!\!\!\int\limits_{\{\ffi>s\}}
\!\!
\frac{  |\na \ffi|_g^{p-3}  
\Big(     (p-1)\,  \nana \ffi (\na\ffi, \!\na\ffi) \, 
- \, \coth(\ffi) \, |\na\ffi|_g^{4}  \Big) }{\sinh(\ffi)}
\,\,\rmd\mu_g .
\end{align}
It is now convenient to set 
\begin{equation}
\label{eq:density}
\mu^{(p)}_{\,g} (E) \,\, =  \, \int\limits_{E} 
\frac{  |\na \ffi|_g^{p-3}  
\Big(     (p-1)\,  \nana \ffi (\na\ffi, \!\na\ffi) \, 
- \, \coth(\ffi) \, |\na\ffi|_g^{4}  \Big) }{\sinh(\ffi)}
%{\rm div}_{\!g}  \bigg( \frac{|\na \ffi|_g^{p-1} \, \na \ffi   }{\sinh(\ffi)} \bigg) 
\,\, \rmd\mu_g \, ,
\end{equation}
for every $\mu_g$-measurable set $E \subseteq \{ \ffi > \ffi_0 + \ep\}$. It is then clear that for $p \geq 1$ the measure $\mu_{\,g}^{(p)}$ is absolutely continuous with respect to $\mu_g$, since $|\na \ffi|_g^{p-3} \,\nana\ffi (\na \ffi, \na\ffi)\leq |\na \ffi|_g^{p-1} |\nana \ffi|_g $.
It is also worth pointing out that, under the hypotheses of Theorem~\ref{thm:refined_conf}, Proposition~\ref{prop:cyl_bis} is in force and thus the same conclusion holds for every $p \geq 0$. 

In view of~\eqref{eq:fip_div}, the function $s \mapsto \Phi_p(s)$ can be interpreted as the repartition function of the measure defined in~\eqref{eq:density}, up to the smooth factor $- \sinh(s)$. Thus, $s \mapsto \Phi_p(s)$ is continuous if and only if the assignment 
$$
s \longmapsto \mu^{(p)}_{\, g} (\{ \ffi > s \})
$$ 
is continuous. Thanks to \cite[Proposition 2.6]{Amb_DaP_Men} and thanks to the fact that $\mu_{\,g}^{(p)}$ is absolutely continuous with respect to $\mu_g$, proving the continuity of the above assignment 
is equivalent to  checking that $\mu_g (\{ \ffi = s \}) = 0$ for every $s>\ffi_0 +\ep$. On the other hand, the Hausdorff dimension of the level sets of $\ffi$ is at most $n-1$, as it follows from the results in~\cite{Hardt_Simon,Lin}. Hence, they are negligible with respect to the full $n$-dimensional measure. This proves the continuity of $\Phi_p$ for $p\geq 1$ under the hypotheses of Theorem~\ref{thm:main_conf} and for every $p \geq 0$ under the hypotheses of Theorem~\ref{thm:refined_conf}.

\subsection{Differentiability.}
\label{sec:diff} 
We now turn our attention to the issue of the differentiability of the functions $s \mapsto \Phi_p(s)$.  
As already observed in the previous subsection, we are assuming that the boundary $\pa M$ is a regular level set of $\ffi$ so that the function $s \mapsto \Phi_p(s)$ can be described in term of an integral depending on the parameter $s$, provided $s \in [\ffi_0, \ffi_0 + 2\ep)$ with $\ep>0$ sufficiently small. Again, the differentiability in the parameter $s$ can be easily checked in this case, using standard results from classical differential calculus. Leaving the details to the interested reader, we pass to consider the case where $s \in (\ffi_0 + \ep, + \infty)$. We start by noticing that for every $p\geq 2$ the function
\begin{equation*}
\frac{  |\na \ffi|_g^{p-4}  
\Big(     (p-1)\,  \nana \ffi (\na\ffi, \!\na\ffi) \, 
- \, \coth(\ffi) \, |\na\ffi|_g^{4}  \Big) }{\sinh(\ffi)}
\end{equation*}
has finite integral in $\{ \ffi > s\}$, for every $s>\ffi_0 + \ep$. Hence, we can apply the coarea formula to expression~\eqref{eq:fip_div}, obtaining
\begin{align}
\label{eq:fip_coa}
\nonumber \Phi_p(s)\,\,& =\,\,
-\, \sinh(s)  \int\limits_{\{\tau>s\}}  \int\limits_{\{\ffi =\tau \}}\!\!\! \frac{(p-1)\, |\na\ffi|_g^{p-4}\, \nana \ffi (\na\ffi, \!\na\ffi) \, 
- \, \coth(\ffi) \, |\na\ffi|_g^{p}}{\sinh(\ffi)}
\,\,\rmd\sigma_{\!g} \,\, \rmd \tau \\
\nonumber & = \,\, 
\sinh(s)  \int\limits_{\{\tau>s\}}  \int\limits_{\{\ffi =\tau \}}\!\!\! \frac{(p-1)\, |\na\ffi|_g^{p-1}\, \Hg \, 
+ \, \coth(\ffi) \, |\na\ffi|_g^{p}}{\sinh(\ffi)}
\,\,\rmd\sigma_{\!g} \,\, \rmd \tau 
\\
&  = \,\, \sinh(s)  \int\limits_{\{\tau>s\}} \!
\Bigg( \,  (p-1)  \!\!\!\! \int\limits_{\{\ffi =\tau \}}\!\!\!\!   \frac{\, |\na\ffi|_g^{p-1}\, \Hg}{\sinh(\tau)} \,\rmd\sigma_{\!g}  \,\, 
+ \,\,  \frac{\coth(\tau)}{\sinh(\tau)}  \, \Phi_p(\tau)
\Bigg)  \,\, \rmd \tau \,,
\end{align}
where in the second equality we have used~\eqref{eq:formula_curvature} and in the third equality we have used the definition of $\Phi_p$ given by formula~\eqref{eq:fip}. By the Fundamental Theorem of Calculus, we have that if the function
\begin{equation*}
%\label{eq:integrand}
\tau \,\, \longmapsto \,\, (p-1) \! \!\!\! \int\limits_{\{\ffi =\tau \}}\!\!\! \!  \frac{\, |\na\ffi|_g^{p-1}\, \Hg}{\sinh(\tau)} \,\rmd\sigma_{\!g}  \,\, 
+ \,\,  \frac{\coth(\tau)}{\sinh(\tau)}  \, \Phi_p(\tau)
\end{equation*}
is continuous, then $\Phi_p$ is differentiable. 
Since we have already discussed in Subsection~\ref{continuity} the continuity of $s \mapsto \Phi_p(s)$,
%is continuous for every $p \geq 1$, 
we only need to discuss
%, for every $p \geq 3$, 
the continuity of the assignment 
\begin{equation}
\label{eq:integrand2}
\tau \,\, \longmapsto  \!\! \int\limits_{\{\ffi =\tau \}}\!\!\! \!  \frac{\, |\na\ffi|_g^{p-1}\, \Hg}{\sinh(\tau)} \,\rmd\sigma_{\!g} \,\, = \!\! \int\limits_{\{\ffi> \tau \}} \!\!
\frac{  |\na \ffi|_\g^{p-3}  
\Big(   \,  \big|\nana \ffi\big|_\g^2
 +  \, (p-3) \, \big| \na |\na \ffi |_\g   \big|_\g^2   \, \Big)       }{\sinh(\ffi)}
\,\,\rmd\mu_g \, .
\end{equation}
We note that the above equality follows from the integral identity~\eqref{eq:id_byparts} and thus from Proposition~\ref{prop:cyl}, which is in force under the hypotheses of Theorem~\ref{thm:main_conf}-(ii), or from Proposition~\ref{prop:cyl_bis}, which is in force under the hypotheses of Theorem~\ref{thm:refined_conf}. In analogy with~\eqref{eq:density} it is natural to set 
\begin{equation*}
\bar\mu^{(p)}_{\,g} (E) \,\, =  \, \int\limits_{E} 
\frac{  |\na \ffi|_\g^{p-3}  
\Big(   \,  \big|\nana \ffi\big|_\g^2
 +  \, (p-3) \, \big| \na |\na \ffi |_\g   \big|_\g^2   \, \Big)       }{\sinh(\ffi)}
\,\, \rmd\mu_g \, ,
\end{equation*}
for every $\mu_g$-measurable set $E \subseteq \{ \ffi > \ffi_0 + \ep\}$. It is now clear that for $p \geq 3$ the measure $\overline{\mu}_{\,g}^{(p)}$ is absolutely continuous with respect to $\mu_g$, and that under the assumptions of Theorem~\ref{thm:refined_conf} the same conclusion holds for every $p \geq 0$. Hence, using the same reasoning as in Subsection~\ref{continuity}, we deuce that the assignment~\eqref{eq:integrand2} is continuous. In turn, we obtain the differentiability of $\Phi_p$ for $p\geq 3$, under the hypotheses of Theorem~\ref{thm:main_conf}, and for every $p \geq 0$, under the hypotheses of Theorem~\ref{thm:refined_conf}. Finally, using~\eqref{eq:fip_coa} and~\eqref{eq:id_byparts}, a direct computation shows that 
\begin{align}
\label{eq:fip_mono}
\nonumber \Phi_p'(s) \,\, 
& =  \,\, - \, (p-1) \!\! \int\limits_{\{\ffi = s \}}\!\!\!\!   {\, |\na\ffi|_g^{p-1}\, \Hg} \,\rmd\sigma_{\!g} \\
& = \,\,- \,(p-1)\,\, \sinh (s) \!\! \int\limits_{\{\ffi> s\}} \!\!
\frac{  \, |\na \ffi|_\g^{p-3}  
\Big(   \,  \big|\nana \ffi\big|_\g^2
 +  \, (p-3) \, \big| \na |\na \ffi |_\g   \big|_\g^2   \, \Big)       }{\sinh(\ffi)}
\,\,\rmd\mu_g \,.
\end{align}
The monotonicity and the rigidity statements in Theorem~\ref{thm:main_conf}-(ii) and Theorem~\ref{thm:refined_conf}-(ii) are now consequences of Proposition~\ref{prop:cyl} and Proposition~\ref{prop:cyl_bis}, respectively.

\subsection{A rigidity result under null Dirichlet boundary conditions.}
To complete our analysis, we need to prove statement (iii) in Theorem~\ref{thm:main_conf}. To this aim, we observe that
\begin{align*}
\frac{\Phi_p'(s)}{s} \,\, & = \,\, - \, (p-1) \,\, \frac{\sinh(s)}{s}\!\! \int\limits_{\{\ffi = s \}}\!\!\!\!   \frac{\, |\na\ffi|_g^{p-1}\, \Hg}{\sinh(s)} \,\rmd\sigma_{\!g} \,\, = \,\, (p-1) \,\, \frac{\sinh(s)}{s}\!\!\int\limits_{\pa M} \!
\frac{|\na \ffi|_g^{p-2} \, \Ric_g (\nu_g, \nu_g)}{\cosh(s)}
\,\,\rmd\sigma_{\!g} \,.
\end{align*}
Taking the limit as $s \to 0^+$ and using~\eqref{eq:id_byparts_ric} in Corollary~\ref{cor:riccinunu} gives~\eqref{eq:der2_fip}. The rigidity statement follows directly from Corollary~\ref{cor:riccinunu}.

\renewcommand{\theequation}{A-\arabic{equation}}
\renewcommand{\thesection}{A}
\setcounter{equation}{0}  % reset counter
\setcounter{theorem}{0}  % reset counter
\section*{Appendix}  
% use *-form to suppress numbering  

The following theorem 
is an extension of the classical Divergence Theorem to the case of open domains whose boundary has a (not too big) nonsmooth portion. Eventhough the argument is quite classical,
we collect the statement and the proof here for
the convenience of the reader.

\begin{theorem}
\label{thm:div}
Let $(M,g)$ be a $n$-dimensional Riemannian manifold, 
with $n\geq 2$,
let $E \subset M$ be a bounded open 
subset of $M$ with compact boundary $\pa E$ of finite $(n-1)$-dimensional Hausdorff measure, 
and suppose that $\pa E = \Gamma \sqcup \Sigma$, where the subsets 
$\Gamma$ and $\Sigma$ have the following properties:
\begin{itemize}
\item[(i)] For every $x \in \Gamma$, there exists an open neighborhood $U_x$ of $x$ in $M$ such that $\Gamma \cap U_x$ is a smooth regular hypersurface. 
\item[(ii)] 
The subset $\Sigma$ is compact and, 
setting $B_\ep (\Sigma) =  \bigcup_{x \in \Sigma} B_\ep (x)$, 
it holds that $\mathscr{H}^{n}(B_\ep(\Sigma)) = o(\ep)$
as $\ep \to 0$.
\end{itemize}
If $X$ is a differentiable vector field defined in a neighborhood of $\overline{E}$ and there exists a positive constant $C>0$ such that
\begin{equation}
\label{eq:reg_X}
\sup_{ \overline{E}} \Big( \, |X| + |{\rm div} X| \, \Big) \,\leq \, C  ,
\end{equation}
%for some constant $C>0$, then
then the following identity holds true
\begin{equation}
\label{eq:thm_div}
\int\limits_E {\rm div} X \, \rmd \mu \, = \, \int\limits_{\Gamma} \langle  X | {\rm n}  \rangle \, \rmd \sigma  ,
\end{equation}
where ${\rm n}$ denotes the exterior unit normal vector field.
\end{theorem}

\begin{proof}
Consider first the trivial case where $X$ vanishes in a 
neighborhood $U$ of $\Sigma$. In this case, let $\tilde E$ 
be a smooth modification of $E$ such that 
$\tilde E\setminus U=E\setminus U$. Hence, we can deduce that
\begin{equation*}
\int\limits_E{\rm div}X\,\rmd\mu
=
\int\limits_{\tilde E}{\rm div}X\,\rmd\mu
=
\int\limits_{\pa\tilde E}
\langle  X | {\rm n}  \rangle \, \rmd \sigma
=
\int\limits_{\Gamma}
\langle  X | {\rm n}  \rangle \, \rmd \sigma,
\end{equation*}
where in the second equality we have used the standard
Divergence Theorem. 
%Thus, \eqref{eq:thm_div} trivially holds. 
To deal with the general case, we now introduce a suitable
approximation $X_{\ep}$ of $X$ which vanishes in a neighborhood of $\Sigma$, whose size tends to zero as $\ep \to 0$. Using the hypothesis (ii), it is not hard to construct a family of cut-off functions $\{\psi_{\ep}\}_{0<\ep\leq 1}\subset{\mathscr C}^1(M)$ with the following properties:
\begin{align*}
{\rm(a)}&\quad 0\leq\psi_{\ep}\leq1\, ,
\phantom{\int\limits}
\\
{\rm(b)}&\quad \psi_{\ep}\equiv 1 \,\,\,\, \mbox{in $B_{\ep}(\Sigma)$} \qquad \hbox{and} \qquad
\psi_{\ep}\equiv0 \,\,\,\, \mbox{in $M\setminus B_{3\ep}(\Sigma)$} \, ,\\
{\rm(c)}&\,\,\,\, \int\limits_{M}\psi_{\ep}\,\rmd\mu \, = \, o(\ep) \hspace{1cm} \hbox{and}
\qquad
\int\limits_{M}|\na\psi_{\ep}|\, \rmd\mu\, = \, o(1)\, , \quad \hbox{as $\ep \to 0$} \, .
\end{align*}
Having the family  $\{\psi_{\ep}\}_{0<\ep\leq 1}$ at hand, we set 
\begin{equation*}
X_{\ep} \, = \, (1-\psi_{\ep})\, X  .
\end{equation*}
 Since by construction
$X_{\ep}$ vanishes in $B_{\ep}(\Sigma)$, 
from previous considerations we have that the theorem
holds true for $X_{\ep}$, that is
\begin{equation}
\label{eq_thm_div_ep}
\int\limits_E {\rm div} X_{\ep} \, \rmd \mu 
\,\, = \, 
\int\limits_{\Gamma} \langle  X_{\ep} | {\rm n}  \rangle \, \rmd \sigma  .
\end{equation}
Observe now that, by hypothesis~\eqref{eq:reg_X} and property (c), one gets
\begin{align*}
\bigg|\int\limits_E {\rm div} X_{\ep} \, \rmd \mu-
\int\limits_E {\rm div} X \, \rmd \mu \, \bigg|
& \leq 
\int\limits_E 
\big| \,
\langle  \na\psi_{\ep} | X \rangle  +   \psi_{\ep}\, {\rm div}X
\, \big| \, \rmd \mu  \leq  C  \bigg(  \int\limits_{M}   |\na\psi_{\ep}|\,\rmd\mu   +  \! \int\limits_{M} \psi_{\ep}  \,\rmd\mu      \bigg)  = o(1),
%\sup_{\overline E}|X|\int\limits_{M}|\na\psi_{\ep}|\,\rmd\mu
%+
%\sup_{\overline E}|{\rm div}X|
%\int\limits_{M}\psi_{\ep}\,\rmd\mu
%\,=\,
%o(1)+o(\ep) \,.
\end{align*}
as $\ep \to 0$. On the other hand, we assumed that the $(n-1)$-dimensional Hausdorff measure of $\pa E$ and thus also the measure of $\Gamma$ is finite. Combining this fact with the hypothesis~\eqref{eq:reg_X}, we deduce that $|\langle  X | {\rm n} \rangle| \in L^1(\Gamma)$. Hence, by the Dominated Convergence Theorem, we conclude that
%\footnote{
%\begin{comm}
%Attenzione!
%\end{comm}
%Per usare il Teorema della Convergenza Dominata qui si
%sta usando il fatto che 
%$\int_{\Gamma}|\langle  \psi_{\ep} X | {\rm n} \rangle|\rmd \sigma
%\leq C\sigma(\Gamma)\leq\sigma(\pa E)$
%e dunque in ultima analisi che $\sigma(\pa E)<\infty$.
%Perch\'e questo \`e vero? Basta la compattezza di 
%$\pa E$ a garantirlo?
%}
%$\psi_{\ep}\to0$ pointwise on $\Gamma$,
%we get
\begin{equation*}
\bigg|
\int\limits_{\Gamma}
\langle  X_{\ep} | {\rm n} \rangle\, \rmd \sigma
-
\int\limits_{\Gamma}
\langle  X | {\rm n} \rangle\, \rmd \sigma
\bigg|
\, = \, 
\bigg|
\int\limits_{\Gamma}
\langle  \psi_{\ep} X | {\rm n} \rangle\, \rmd \sigma
\bigg| \, \leq \, \int\limits_{\Gamma}
\psi_{\ep} \, | \langle   X | {\rm n} \rangle| \, \rmd \sigma
\, \longrightarrow \, 0 \, ,
\end{equation*}
since $\psi_{\ep}\to0$ pointwise on $\Gamma$, as $\ep \to 0$. This limit and the previous estimate show that
\eqref{eq:thm_div} can be derived in the limit
as $\ep\to 0$ from \eqref{eq_thm_div_ep}.
This concludes the proof of the theorem. 
\end{proof}

\subsection*{Acknowledgements}
\emph{
V.~A. has received funding from the 
European Research Council / ERC Advanced Grant ${n^o}\ 340685$.
L.~M. has been partially supported by the Italian project FIRB 2012 ``Geometria Differenziale e Teoria Geometrica delle Funzioni'' as well as by the SNS project ``Geometric flows and related topics''.
The authors are members of the Gruppo Nazionale per
l'Analisi Matematica, la Probabilit\`a e le loro Applicazioni (GNAMPA)
of the Istituto Nazionale di Alta Matematica (INdAM).
}

%%%%%%%%%%%%%%%%%%%%%%%%%%%%%%%%%%%%%%%%%%%%%%%
%%%%%%%%%%%%%%%%%%%%%%%%%%%%%%%%%%%%%%%%%%%%%%%

\bibliographystyle{plain}

\end{document}